\documentclass[12pt]{amsart}
\usepackage{amssymb}
\usepackage{amscd}
\usepackage[bookmarks=true]{hyperref}
\hypersetup{colorlinks,citecolor=black,linkcolor=black}

\usepackage{graphicx}
\usepackage{amsmath}
\usepackage{amsthm}
\usepackage{amsfonts}
\usepackage{graphicx}
\usepackage{tikz}
\usepackage{extarrows}
\usepackage{comment}
\usepackage{mathtools}
\usepackage{enumitem}
\usepackage[noadjust]{cite}
\usepackage{cases}
\usepackage{xcolor}
\usepackage{mathrsfs}

\newtheorem{theorem}{Theorem}[section]
\newtheorem{lemma}[theorem]{Lemma}
\newtheorem{prop}[theorem]{Proposition}
\newtheorem{cor}[theorem]{Corollary}
\theoremstyle{definition}
\newtheorem{definition}[theorem]{Definition}
\newtheorem{rmk}[theorem]{Remark}
\newtheorem{example}[theorem]{Example}

\newtheorem*{conjecture*}{Conjecture}

\newtheoremstyle{TheoremNum}
        {9pt}{9pt}              
        {\itshape}                      
        {}                              
        {\bfseries}                     
        {.}                             
        { }                             
        {\thmname{#1}\thmnote{ \bfseries #3}}
\theoremstyle{TheoremNum}
\newtheorem{repthrm}{Theorem}

\newcommand\Hy{\mathbb{H}}

\newcommand\Z{\mathbb{Z}}
\newcommand\Q{\mathbb{Q}}
\newcommand\R{\mathbb{R}}
\newcommand\C{\mathbb{C}}
\newcommand\cO{\mathcal{O}}

\newcommand\cQ{\mathcal{Q}}

\newcommand\area{\mathrm{area}}
\newcommand\sys{\mathrm{sys}}
\newcommand\vol{\mathrm{vol}}
\DeclareMathOperator{\SL}{SL}

\DeclareMathOperator{\PSL}{PSL}

\newcommand{\Sph}{\mathbb{S}}
\newcommand{\N}{\mathbb{N}}
\newcommand{\F}{\mathbb{F}}

\newcommand{\frakR}{\mathfrak{R}}
\newcommand{\rO}{\mathcal{O}}

\newcommand{\cF}{\mathcal{F}}
\newcommand{\cH}{\mathcal{H}}

\newcommand{\cK}{\mathcal{K}}
\newcommand{\cT}{\mathcal{T}}
\newcommand{\fT}{\mathfrak{T}}
\newcommand{\rT}{\mathrm{T}}
\newcommand{\defeq}{\vcentcolon=}

\newcommand{\PSLTR}{\mathrm{PSL}({2},{\mathbb{R}})}

\newcommand{\PGLTR}{\mathrm{PGL}({2},{\mathbb{R}})}

\newcommand{\SLTR}{\mathrm{SL}({2},{\mathbb{R}})}

\newcommand{\tr}{\mathrm{tr \,}}
\newcommand{\Tr}{\mathrm{Tr \,}}

\newcommand{\inv}{^{-1}}

\newcommand{\minus}{\hspace{-2pt}\setminus\hspace{-2pt}}
\newcommand{\Gammat}{\Gamma^{(2)}}

\newcommand{\T}{\mathrm{Teich}}
\newcommand{\Ht}{\mathrm{H}}
\newcommand{\vp}{v_{\mathfrak{p}}}

\begin{document}
\title{Closed geodesics on Semi-Arithmetic Riemann Surfaces}


\author{Gregory Cosac}\thanks{The first author was supported by CNPq-Brazil research grant 141204/2016-8}
\author{Cayo D\'oria} \thanks{D\'oria is grateful for the support of FAPESP grant 2018/15750-9.}
\address{
IMPA\\
Estrada Dona Castorina, 110\\
22460-320 Rio de Janeiro, Brazil}
\email{cosac@impa.br}
\address{
Departamento de Matem\'atica Aplicada, IME-USP\\ Rua do Mat\~ao, 1010, Cidade Universit\'aria\\
05508-090, S\~ao Paulo SP, Brazil.}
\email{cayofelizardo@ime.usp.br}


\begin{abstract}
In this article, we study geometric aspects of semi-arithmetic Riemann surfaces by means of number theory and hyperbolic geometry. First, we show the existence of infinitely many semi-arithmetic Riemann surfaces of various shapes and prove that their systoles are dense in the positive real numbers. Furthermore, this leads to a construction, for each genus $g\geq2$, of infinite families of semi-arithmetic surfaces with pairwise distinct invariant trace fields, giving a negative answer to a conjecture of B. Jeon. Finally, for any semi-arithmetic surface we find a sequence of congruence coverings with logarithmic systolic growth and, for the special case of surfaces admitting modular embedding, we are able to exhibit explicit constants. 
\end{abstract}

\maketitle

\section{Introduction}

A \emph{closed Riemann surface} is a compact $1$-dimensional complex manifold without boundary. Due to Riemann's celebrated \emph{Uniformization Theorem}, any closed Riemann surface of genus $g\geq2$ may be identified with the quotient $\Gamma \backslash \Hy$ of the upper half-plane $\Hy=\{x+iy \in \mathbb{C} \mid y>0 \}$ by the action of a discrete group $\Gamma$ of automorphisms of $\Hy$. Since the Lie group $\PSLTR$, acting on $\Hy$ via M\"obius transformations, is isomorphic to the group of automorphisms of $\Hy$, $\Gamma$ can be realised as a \emph{Fuchsian group}, i.e, a discrete subgroup of $\PSLTR$. Furthermore, it is a fortuitous coincidence that $\PSLTR$ is also isomorphic to the group of orientation-preserving isometries of $(\Hy,d_{\Hy})$, where $d_{\Hy}$ is the distance function on $\Hy$ induced by the hyperbolic metric $ds^2=\frac{dx^2+dy^2}{y^2}$. Therefore, any such surface is naturally equipped with a hyperbolic structure and, for that reason, is called a \emph{hyperbolic surface}.

Let $\rT_g$ denote the \emph{Teichm\"uller space of genus $g$}, i.e., the space of all isomorphism classes of \emph{marked} closed Riemann surfaces of genus $g$ (we shall always assume $g \ge 2$). Among the most studied invariants used in order to understand the hyperbolic structure of a generic surface $S\in \rT_g$, one finds the diameter $\mathrm{diam}(S)$ of $S$, the spectral gap $\lambda_1(S)$ of the Laplace-Beltrami operator defined on $S$ and, more generally, the distribution of such spectrum (see \cite{Luo94}), the isometry group $\mathrm{Isom}(S)$ of $S$, and the systole $\sys(S)$, defined as the minimal length of a closed geodesic of $S$. 

The space $\rT_g$ turns out to be quite large. Indeed, $\rT_g$ may be equipped with a topology that makes it a topological manifold of dimension $6g-6$. In view of this fact, one may endeavour to study special classes of hyperbolic Riemann surfaces endowed with some extra structure that allows the application of a greater number of techniques. For us, the starting point is the class of \emph{arithmetic} Riemann surfaces. This class contains known surfaces such as the Hurwitz surface of genus $3$ and the Bolza surface of genus $2$. Moreover, in several aspects, arithmetic surfaces present very rare or strong properties which are not known to hold for a generic hyperbolic Riemann surface. To mention a few:

\begin{example}\label{040320.1}
	\hfill
	
	\begin{enumerate}[label=(\roman*)]
		\item Given an arithmetic Riemann surface $S$, it is possible to find a family $S_i$ of finite \emph{congruence} coverings of $S$ with arbitrarily large degree forming a family of \emph{geometric expanders} with explicit uniform spectral gap (see \cite{BB11}), namely $$\lambda_1(S_i) > \frac{5}{16}.$$
		\item The same family in the previous item has the explicit logarithmic systolic growth (see \cite{BS94}, \cite{KSV07})
		\[\sys(S_i) \gtrsim \frac{4}{3} \log(\area(S_i)),\]
		\item The commensurability class of an arithmetic Riemann surface is determined by its Laplace-Beltrami spectrum \cite{Reid92}.
		\item The cofinite Fuchsian group of minimal coarea is the triangular group $(2,3,7)$ which is arithmetic by the Takeuchi list in \cite{Takeuchi77}. 
	\end{enumerate}
\end{example}

On the other hand, arithmetic surfaces are somewhat rigid in the sense that, for any $g \ge 2$, the number of arithmetic surfaces in the \emph{moduli space} $\mathrm{M}_g$ is finite (see \cite{Bor81}). In \cite{SW}, Schmutz Schaller and Wolfart extended the notion of arithmeticity to a larger class of \emph{semi-arithmetic} Riemann surfaces. This new class contains all arithmetic surfaces as well as the important class of \emph{quasiplatonic} surfaces (see \cite{PW19}). Moreover, Schaller and Wolfart give infinite families of examples which are not arithmetic.

Since then,  
some relevant facts concerning the arithmetic and algebraic aspects of semi-arithmetic surfaces have been deduced (see \cite{Gen12}, \cite{K15}). The main purpose of our work is to give an account of the geometric aspects of semi-arithmetic Riemann surfaces by means of number theory and hyperbolic geometry.

Our first theorem reveals the discrepancy between the semi-arithmetic and the arithmetic classes. Before we continue, let us recall the definition of a length function. Given a closed topological surface $F$ of genus $g \ge 2$ and a homotopically nontrival closed curve $\alpha \subset F$, we define the corresponding \emph{length function} $\ell_\alpha:\rT_g \rightarrow \mathbb{R},$ that associates to each Riemann surfaces $S$ in $\rT_g$ the length $\ell_\alpha(S)$ of the unique closed geodesic on $S$ freely homotopic to $\alpha$.
While for any $g$ and any length function on $\rT_g$ the image of the set of arithmetic surfaces is discrete, we prove the following:

\begin{theorem}\label{dense:thm}
	For any $g \ge 2$ there exists a length function $\ell:\rT_g \to \mathbb{R}$ such that 
	\[\{ \ell(S) \mid S \in \rT_g \mbox{ is semi-arithmetic} \}\]
	is dense on the set of positive real numbers.
\end{theorem}

The \emph{short geodesic conjecture} predicts the existence of a universal lower bound for the systole of any arithmetic Riemann surface. Theorem \ref{dense:thm} implies that this conjecture cannot be true for the more general class of semi-arithmetic closed surfaces. In fact, with some more effort, we can prove that:

\begin{theorem}\label{cor:densityofsys}
	The set $\{ \sys(S) \mid S \mbox{ is a closed semi-arithmetic Riemann surface} \}$ is dense in the positive real numbers.
\end{theorem}

An important feature of arithmetic Riemann surfaces, however, remains valid for the semi-arithmetic class. Indeed, since it still makes sense to define congruence coverings for semi-arithmetic Riemann surfaces (see \cite{K15}), we prove that item (ii) in Example \ref{040320.1} extends to closed semi-arithmetic Riemann surfaces, with a possibly weaker constant. 

\begin{theorem}\label{sys:thm}
	Let $S$ be a closed semi-arithmetic Riemann surface. There exists a sequence of congruence coverings $S_i \to S$ of arbitrarily large degree such that
	$$\sys(S_i) \geq \lambda \log(\area(S_i))-c,$$
	where $\lambda>0$ and $c$ are constants depending only on $S$.
\end{theorem}

Moreover, for a suitable subclass of semi-arithmetic surfaces, the Riemann surfaces admitting a \emph{modular embedding}, we are able to obtain an explicit constant $\lambda$ depending only on the corresponding embedding. For any $r \geq 1$, consider the symmetric space $X_r=\Hy^r$, i.e., the product of $r$ copies of $\Hy$. We say that a Riemannian manifold is an $X_r$-manifold if it is locally isometric to $X_r$.

\begin{theorem}\label{modemb:thm}
	If $S$ is a semi-arithmetic Riemann surface which admits modular embedding into an arithmetic $X_r$-manifold for some $r \ge 1$, then there exists a sequence of congruence coverings $S_i \to S$ of arbitrarily large degree satisfying
	$$\sys(S_i) \geq \frac{4}{3r} \log(\area(S_i))-c,$$
	where the constant $c$ does not depend on the $i$.
\end{theorem}

Note that when $r=1$ we recover the known result for arithmetic surfaces mentioned in Example \ref{040320.1}.

Theorem \ref{modemb:thm} will follow from geometric results combined with estimates for the systolic growth of congruence coverings of irreducible arithmetic $X_r$-manifolds, for any $r \geq 1$. In \cite{Murillo17}, Murillo treats the noncompact case of such arithmetic manifolds, but his argument may be generalised almost verbatim for the compact case. For the sake of completeness, in Appendix \ref{pliniogeral} we provide a proof of:

\begin{theorem} \label{plinio:appendix}
	Let $M$ be an arithmetic closed $X_r$-manifold, there exists a sequence $M_i$ of congruence coverings of $M$ with $\vol(M_i) \to \infty$ such that $$\sys_1(M_i) \ge \frac{4}{3\sqrt{r}}\log(\vol(M_i)) -c,$$
	where $c$ does not depend on $i$.
\end{theorem}

We conclude the introduction with a brief outline of the article. In Section \ref{back} we review the necessary basic mathematical tools that will be used later on. In \S\ref{numbertheory}, a minimal list of definitions in algebraic number theory is given. We recall some general facts about Fuchsian groups and their Teichm\"uller spaces, and define semi-arithmetic Riemann surfaces in \S\ref{semiFuchs}. 
In \S\ref{congruencesub}, the notion of congruence subgroups in the context of semi-arithmetic Fuchsian groups is introduced. 

Theorem \ref{dense:thm} is proved in Section \ref{densetoflengths}. The idea for the proof is to embed surface-groups in groups generated by reflections across the sides of some specific hyperbolic polygons, namely, \emph{trirectangles} and \emph{right-angled hexagons}. These embeddings are constructed in such a way that their images contain a special element whose trace is a free parameter. By varying this parameter, we are able to obtain the density asserted in the theorem.

In \S \ref{19420.4} we give two applications of Theorem \ref{dense:thm}. First, using the density of lengths for semi-arithmetic surfaces of genus $2$, we conduct an argument similar to the one used in \cite{D17} and show that any such length is the systole of a finite covering of the corresponding surface, thus proving Theorem \ref{cor:densityofsys}. Next, we use the methods applied in the proof of Theorem \ref{dense:thm} in order to construct (see Theorem \ref{20420.1}), for each genus $g\geq2$, infinite families of semi-arithmetic Riemann surfaces with pairwise distinct invariant trace field. In \cite[Conjectrue 2]{Jeon19}, B. Jeon conjectured that, for a fixed genus $g\geq2$, only finitely many number fields and quaternion algebras could arise as the invariant trace field and quaternion algebra of a genus $g$ hyperbolic Riemann surfaces \emph{with integral traces}. So, in particular, we provide a negative answer to Jeon's conjecture.

In Section \ref{18220.1}, we prove Theorems \ref{sys:thm} and \ref{modemb:thm} . Theorem \ref{plinio:appendix} will follow from a more general theorem whose proof is provided in Appendix \ref{pliniogeral}. The proofs of these theorems are, in some sense, reminiscent of an idea of Margulis used in \cite{M82}. We estimate the displacement of any closed geodesic in the congruence coverings by means of a convenient arithmetic function in each context. On the other hand, the asymptotic behavior of the area/volume growth of the congruence coverings follows from recent results in \cite{K15}.

\vskip 10pt
\noindent{\textbf{Acknowledgements.}}
We would like to thank Professor Mikhail Belolipetsky and Plinio Murillo for their comments and
suggestions on the original manuscript.

\section{Background}\label{back}

\subsection{Number theoretic tools}\label{numbertheory}

In this subsection we review some classical definitions and results which will be assumed along the text. This material can be found in any textbook of algebraic number theory or arithmetic geometry (e.g. \cite{Lang} and \cite{gtmReid03}).

A complex number $z$ is said to be an \emph{algebraic number} if there exists a non-zero polynomial $P \in \mathbb{Z}[X]$ such that $P(z)=0.$ When $P$ can be taken to be monic, i.e. with leading coefficient $1$, then we say that $z$ is an \emph{algebraic integer}. Let $X \subset \C$ be a nonempty set, we define $\Q(X)$ as the intersection of all subfields of $\C$ containing $X$. Note that $\Q(X)$ is a field. 

A \emph{number field} is a subfield $L \subset \C$ which is a finite extension of $\Q$, i.e. $L$ is a $\Q$-vector space of finite dimension. When $\alpha$ is an algebraic number, $\Q(\alpha)$ is a number field. Indeed, let $P$ be \emph{minimal polynomial} of $\alpha$, which, by definition, means that $P$ is the unique monic polynomial over $\Q$ of minimal degree such that $P(\alpha)=0$. If $n$ is equal to the degree of $P$, it is easy to check that $\Q(\alpha)$ is a finite extension of $\Q$ and $\{1,\alpha,\dots,\alpha^{n-1}\}$ is a $\Q$-basis of $\Q(\alpha)$. Conversely, if $L$ is a number field, then there exists $\beta \in \C$ such that $L=\Q(\beta)$. Although $\beta$ is not unique, any field homomorphism  $\sigma:L \rightarrow \C$ is determined by $\sigma(\beta)$, which is a root of the minimal polynomial of $\beta$. Hence, for a finite extension $L$ there exist exactly $n$ monomorphisms $\sigma:L \to \C$, where $n=\mathrm{dim}_\Q(L)=:[L:\Q]$ is called the \emph{degree} of $L$. These monomorphisms are called the \emph{Galois embeddings} of $L$. We say that $\beta \in \C$ is Galois conjugated with an algebraic number $\alpha$ if there is a Galois embedding $\sigma:\Q(\alpha) \to \C$ such that $\beta=\sigma(\alpha)$. In particular, $\beta$ is algebraic.

An algebraic number $\alpha$ is said to be \emph{totally real} (resp. \emph{positive}) if $\alpha$ and all its Galois conjugates are real (resp. positive). Note that if $\alpha$ is totally real then $\alpha^2$ is totally positive. We record the following easy lemma for future use:

\begin{lemma}\label{10320.1}
	Let $\alpha$ be a totally positive algebraic integer. Then $\sqrt{\alpha}$ is a totally real algebraic integer.
\end{lemma}

\begin{proof}
	Let $P(X) \in \Q[X]$ be the minimal polynomial of $\alpha$. Being $\alpha$ an algebraic integer, it is known that $P$ is actually a monic polynomial with coefficients in $\Z$. Now, $\sqrt{\alpha}$ is a root of $P(X^2)$ which makes it also an algebraic integer. Moreover, we may factor $P$ as $P(X) = \prod_i(X - \alpha_i)$ where $\alpha_i$ denotes the Galois conjugates of $\alpha$. It follows that $P(X^2) = \prod_i(X^2 - \alpha_i) = \prod_i(X - \sqrt{\alpha_i})(X + \sqrt{\alpha_i})$. Since all $\alpha_i$ are positive real numbers, $\pm\sqrt{\alpha_i}$ are real numbers, among which are all the Galois conjugates of $\sqrt{\alpha}$.
\end{proof}

Let $L$ be a number field, the subset of algebraic integers of $L$ form a ring $\mathcal{O}_L$ known as the \emph{ring of integers} of $L$. Let $\sigma:L \to \C$ be a Galois embedding. We say that $\sigma$ is real if $\sigma(L) \subset \R$, otherwise we say that $\sigma$ is complex. Note that if $\sigma$ is complex, then the complex conjugated $\bar{\sigma}$ is another Galois embedding different from $\sigma$. Let $r_1$ be the number of real Galois embeddings of $L$ and $r_2$ the number of pairs of complex conjugate embeddings, so $[L:\Q] = r_1 + 2r_2$. The following is a well-known theorem in algebraic number theory.

\begin{theorem}[Dirichlet's Unit Theorem, \cite{Neukirch99}]\label{dirichlet}
	The group of units $\mathcal{O}_L^\times$ of $\mathcal{O}_L$ is a finitely generated group of rank $r_1 + r_2 -1$. The torsion part is the cyclic subgroup formed by the roots of unity contained in $L$.
\end{theorem}

A \emph{valuation} $v$ on $L$ is a multiplicative sub-additive nonnegative function $v: L \to \R$ such that $v(x)=0$ if and only if $x=0$. Two valuations $v_1,v_2$ on $L$ are said to be equivalent if $v_1 = v_2^a$ for some positive number $a\in\R$. If a valuation $v$ satisfies a condition stronger than sub-additivity, namely, that
\begin{align}\label{19420.1}
	v(x+y) \leq \max \{v(x), v(y)\} \quad \text{for all } x,y\in L,
\end{align} 
we say $v$ is \emph{non-Archimedean}. On the other hand, it is \emph{Archimedean} if it is not equivalent to a valuation satisfying \eqref{19420.1}. Next, we indicate two ways of describing a valuation on the number field $L$:

\begin{enumerate}[label=(\roman*)]
	\item Let $\sigma: L \to \C$ be a Galois embedding and define $v_{\sigma}(x) = |\sigma(x)|, \ \forall x\in L$, where $|\cdot|$ is the usual norm in $\C$. Then $v_{\sigma}$ is Archimedean and two such valuations $v_{\sigma}, v_{\sigma'}$ are equivalent if and only if $\sigma'$ is the complex conjugate of $\sigma$. \label{19420.2}
	\item Let $\mathfrak{p}$ be any prime ideal in $\cO_{L}$ and let $N(\mathfrak{p})$ denote its \emph{norm}, i.e, the cardinality of the finite field $\cO_{L}/\mathfrak{p}$. For any nonzero $x\in\cO_{L}$, define $\vp(x) = N(\mathfrak{p})^{-\mathrm{ord}_{\mathfrak{p}}(x)}$, where $\mathrm{ord}_{\mathfrak{p}}(x)$ is the largest integer $k$ such that $\mathfrak{p}^k$ divides $x\cO_K$. It is natural to extend these functions to 0 as $\mathrm{ord}_{\mathfrak{p}}(0) = +\infty$ and $\vp(0) = 0$. Finally, since $L$ is the field of fractions of $\cO_{L}$, we may extend $\vp$ to $L$ by imposing the multiplicative property and defining $\vp(x/y)\defeq\vp(x)/\vp(y)$. It is easy to check that $\vp$ defines a non-Archimedean valuation on $L$. \label{19420.3}	
\end{enumerate}

It is an important result that these are all valuations on $L$, up to equivalence. In other words, any Archimedean valuation on $L$ is equivalent to $v_{\sigma}$ for some Galois embedding $\sigma: L \to \C$, and any non-Archimedean valuation on $L$ is equivalent to $\vp$ for some prime ideal $\mathfrak{p}\in \cO_{L}$. 

We define a \emph{place} of $L$ to be an equivalence class of valuations on $L$ and denote by $V(L)$ the set of all places of $L$. The set $V(L)$ decomposes as $V_{\infty}\cup V_{f}$, where $V_{\infty}$ denotes the set of Archimedean (or \emph{infinite}) places of $L$ and $V_{f}$, the set of non-Archimedean (or \emph{finite}) places of $L$. Note that $V_\infty$ has cardinality $r_1+r_2$, whereas $V_{f}$ is in bijection with the set of prime ideals of $\cO_{L}$.

We adopt the following convention: for an infinite place associated to a real Galois embedding $\sigma$, we consider $v_\sigma$ as described in \ref{19420.2} to be the normalised valuation in its equivalence class. In case $\sigma$ is a complex embedding, choose $v_\sigma(x) = |\sigma(x)|^2$ instead. Finally, for a finite place associated to $\mathfrak{p}$, we consider $\vp$, as defined in \ref{19420.3}, to be the normalised valuation. Henceforth, we will usually denote an arbitrary place of $L$ by $v$, as well as the normalised valuation in its equivalence class following the aforementioned criteria.

For any $\alpha \in L$, $\alpha \neq 0$, there exists at most finitely many places $v\in V(L)$ such that $v(\alpha)>1$. Thus, we may define
\[\Ht(\alpha)= \left( \prod_{v \in V(L)} \max\{1,v(\alpha)\} \right)^{\frac{1}{n}}.\]

We say that $\Ht(\alpha)$ is the absolute \emph{height} of $\alpha$. The name absolute comes from the fact that $\Ht$ does not depend on the field containing $\alpha$ (cf. \cite[Chapter 3]{Lang}). Note that, if $\alpha, \beta \in L$, we have that $v(\alpha + \beta)\leq 4\max{v(\alpha),v(\beta)}$ for $v$ an Archimedean place and $v(\alpha+\beta)\leq \max{v(\alpha),v(\beta)}$ for $v$ a non-Archimedean place, whence the function $\Ht$ satisfies the following inequalities:

\begin{align}
 \Ht(\alpha+\beta) & \leq 4\Ht(\alpha)\Ht(\beta) \label{sumh};\\ 
 \Ht(\alpha \beta) & \leq \Ht(\alpha)\Ht(\beta) \label{prodh}.
\end{align}
%

Now let $A=\begin{pmatrix} a & b \\ c & d
 \end{pmatrix}$ be a matrix with $a,b,c,d \in L$ and $ad-bc=1$. We define the height of $A$ by
 \[\Ht(A)= \left( \prod_{v \in V(L)} \max\{1,v(a), v(b), v(c), v(d) \} \right)^{\frac{1}{d}} .\]
Note that $\Ht(A)$ is well-defined and it holds that $\Ht(A)=\Ht(-A)$. So, for $\gamma \in \PSL(2,L)$ we may define $\Ht(\gamma):=\Ht(A)$, for any representative $A$ of $\gamma$. Again the definition of $\Ht(\gamma)$ does not depend on the field $L$. It is immediate from the definition that properties \eqref{sumh} and \eqref{prodh} extend to the following property:
\begin{align}\label{prodH}
    \Ht(\gamma \eta) \leq 4 \Ht(\gamma)\Ht(\eta) \quad \mbox{ for any } \gamma,\eta \in \PSL(2,L). 
\end{align}

\subsection{Semi-arithmetic Fuchsian groups}\label{semiFuchs}
Let $\Gamma < \PSLTR$ be a finitely generated nonelementary Fuchsian subgroup of the first kind. Such groups are known to have the following \emph{standard presentation} (see, for example, \cite{Greenberg77}):

\begin{align}\label{201019.1}
&\left\langle \vphantom{\prod_{i=1}^g} A_1,B_1,\dots, A_g,B_g,C_1,\dots,C_r,P_1,\dots,P_s \ \Big| \right. \nonumber\\
& \hspace{3cm} \left. C_1^{m_1} = \dots = C_r^{m_r} = \prod_{i=1}^g [A_i,B_i] \, C_1 \cdots C_r \, P_1 \cdots P_s=1 \right\rangle,
\end{align}
where $A_1,B_1,\dots, A_g,B_g$ are hyperbolic elements, $C_1,\dots,C_r$ are elliptic elements of order $m_1,\dots,m_r$, respectively, and $P_1,\dots,P_s$ are parabolic elements of $\PSLTR$. The quotient $\Gamma\backslash\Hy$ has a unique Riemann surface structure that makes the projection $\Pi \! : \Hy \to \Gamma\backslash\Hy$ holomorphic. With this structure, $\Gamma\backslash\Hy$ becomes a genus $g$ Riemann surface with $s$ punctures and $r$ marked points over which $\Pi$ is branched (i.e., $\Pi$ is a branched covering). Alternatively, $\Gamma\backslash\Hy$ has the structure of a hyperbolic 2-orbifold with genus $g$, $s$ punctures and $r$ cone singularities. Its hyperbolic area is given by

\begin{align}\label{241019.1}
\area(\Gamma\backslash\Hy) = 2\pi \left[ 2g-2 + s + \sum_{i=1}^r \left( 1 - \frac{1}{m_i} \right) \right]
\end{align}
and we say $\Gamma$ is \emph{cofinite} in case $\area(\Gamma\backslash\Hy)<\infty$.

Conversely, when $g\geq0,r\geq0,m_i\geq2,i=1,\dots r,s\geq0$ are integers such that the right-hand side of \eqref{201019.1} is positive, then Poincar\'e's Theorem asserts the existence of a Fuchsian group with signature $(g;m_1,\dots,m_r;s)$.

In what follows, unless otherwise stated, we will assume $\Gamma$ to be \emph{cocompact}, i.e., such that $\Gamma\backslash\Hy$ is compact. In this case, $s=0$ and we denote the signature simply as $(g;m_1,\dots,m_r)$.

By a \emph{representation} of $\Gamma$ into $\PSLTR$ we mean a group homomorphisms $\phi:\, \Gamma \to \PSLTR$. Note that $\phi$ is determined by the image of any set generating $\Gamma$. Moreover, the representations of $\Gamma$ into $\PSLTR$ are in one-to-one correspondence with the elements $(A_1,B_1,\dots,A_g,B_g,C_1,\dots,C_r)$ of $\PSLTR^{2g+r}$ whose coordinates satisfy the relations described in \eqref{201019.1}. Let $\frakR'(\Gamma)$ denote the set of all representations of $\Gamma$ into $\PSLTR$. Given the above correspondence, we will identify $\frakR'(\Gamma)$ with a closed subset of $\PSLTR^{2g+r}$. This identification induces a natural topology on $\frakR'(\Gamma)$. Let $\frakR(\Gamma)$ be the subset of $\frakR'(\Gamma)$ consisting of injective representations $\phi$ such that $\phi(\Gamma)$ is a discrete cocompact subgroup of $\PSLTR$. In \cite{Weil60}, A. Weil proved that $\frakR(\Gamma)$ is open in $\frakR'(\Gamma)$. This is also true for noncocompact Fuchsian groups or even Fuchsian groups of the second kind, although, in those cases, it follows from a stronger result due to Bers (cf.\cite{Bers70}).

Note that $\PGLTR$, the group of all conformal and anti-conformal homeomorphisms of the upper half-plane $\Hy$, acts on $\frakR(\Gamma)$ by conjugation. The \emph{Teichm\"uller space} of $\Gamma$, $\T(\Gamma)$, is defined to be the quotient of $\frakR(\Gamma)$ by this action.

\begin{rmk}\label{15320.1}
	Let $S_g$ be a closed orientable surface of genus $g$. The Teichm\"uller space $\T(S_g)$ of $S_g$ may be defined as the set of all hyperbolic metrics on $S_g$ up to isometries isotopic to the identity. Alternatively, one may consider all pairs formed by a genus $g$ Riemann surface together with a \emph{marking}, i.e, a choice of homotopy classes for the canonical generators of its fundamental group. Two such pairs are said to be equivalent when there exists a biholomorphism between the surfaces that respects the corresponding markings. The space of equivalence classes is called the \emph{Teichm\"uller space of genus g}, and is denoted by $\mathrm{T}_g$. We observe that the spaces $\T(S_g)$ and $\mathrm{T}_g$ may be identified. Moreover, when $S_g=\Gamma\backslash\Hy$, there exits a natural correspondence between $\T(\Gamma)$ and $\T(S_g)\cong\mathrm{T}_g$. For further discussion on this topic, one may refer to well known textbooks on the subject, such as \cite{Imayoshi12} and \cite{Farb11}.
\end{rmk}

For an element $A\in \PSLTR$ we define its trace, $\Tr A$, to be $|\tr\tilde{A}|$, where $\tr$ denotes the usual trace of matrices and $\tilde{A}$ is any one of the pre-images of $A$ under the projection $\mathrm{P}\! : \SLTR \to \PSLTR$. Let $\{\Tr\gamma \mid \gamma\in\Gamma \}$ denote the \emph{trace set} of $\Gamma$. The field extension $\Q(\Tr\Gamma)$ obtained by adjoining all the traces of elements of $\Gamma$ to the rational number field $\Q$ is called the \emph{trace field} of $\Gamma$. If $\Gammat$ denotes the subgroup of $\Gamma$ generated by $\{\gamma^2 \mid \gamma\in\Gamma \}$, then $\Q(\Tr\Gammat)$ is called the \emph{invariant trace field} of $\Gamma$ and is denoted by $k\Gamma$. The name is justified since $k\Gamma$ is an invariant of the commensurability class of $\Gamma$ (see \cite[Theorem 3.3.4]{gtmReid03}), where two groups $\Gamma_1$ and $\Gamma_2$ are said to be \emph{commensurable} if $\Gamma_1\cap\Gamma_2$ is a subgroup of finite index in both $\Gamma_1$ and $\Gamma_2$. More generally, two groups are said to be \emph{commensurable in the wide sense} if one is commensurable to a conjugate of the other. Since the trace of a matrix is invariant under conjugation, it follows that $k\Gamma$ is also invariant under wide commensurability.


In \cite{Takeuchi75}, Takeuchi characterised \emph{arithmetic Fuchsian groups} as cofinite Fuchsian groups satisfying the following two conditions (cf. \cite{SW}):
\begin{enumerate}[label=(\roman*)]
	\item $k\Gamma$ is an algebraic number field and $\Tr\Gammat \subset \rO_{k\Gamma}$; \label{211019.1}
	\item If $\phi: k\Gamma \to \C$ is any Galois embedding different from the identity then $\phi(\Tr\Gammat)$ is bounded in $\C$. \label{211019.2}
\end{enumerate}

One possible way to define a class of Fuchsian groups larger than that of the arithmetic groups is to drop the hypothesis on the boundedness of $\phi(\Tr\Gammat)$. Note, however, that conditions \ref{211019.1} and \ref{211019.2} together imply that $k\Gamma$ is \emph{totally real} (\cite[Proposition 2]{Takeuchi75}), i.e., that every embedding of $k\Gamma$ into $\C$ has its image lying in $\R$. This property is retained in the following definition:

\begin{definition}[\hspace{-1.1 mm} \cite{SW}] \label{DDay19.2}
	A cofinite Fuchsian group $\Gamma$ is said to be \emph{semi-arithmetic} when $k\Gamma = \Q(\Tr\Gamma^{(2)})$ is a totally real number field and $\Tr\Gamma^{(2)}\subset \rO_{k\Gamma}$.
\end{definition}

Equivalently, it follows from elementary trace relations that $\Gamma$ is semi-arithmetic if $k\Gamma$ is a totally real number field and, for every $\gamma\in\Gamma$, $\Tr\gamma$ is an algebraic integer. 

As per usual, we say that a hyperbolic Riemann surface (or hyperbolic 2-orbifold) $S$ is semi-arithmetic in case $S=\Gamma\backslash\Hy$ where $\Gamma<\PSLTR$ is semi-arithmetic. Likewise, we may refer to $k\Gamma$ as the invariant trace field of $S$.

The first examples are, of course, arithmetic Fuchsian groups. Triangle groups are also semi-arithmetic, as it follows from the description of their trace sets and invariant trace fields given by Takeuchi in \cite{Takeuchi77}. Furthermore, Takeuchi's results imply that all but finitely many triangle groups are \emph{strictly semi-arithmetic}, i.e., nonarithmetic semi-arithmetic groups. In \cite{SW}, Schaller and Wolfart describe infinite families of semi-arithmetic groups not admitting modular embedding (see $\S$\ref{18220.1} for definition).

One may easily verify that, for a finite-index subgroup $\Gamma' < \Gamma$, if $\Tr\Gamma'$ is contained in the ring of algebraic integers of $\C$, then the same is true for $\Tr\Gamma$. It follows from this observation that being semi-arithmetic is invariant under commensurability in the wide sense.

In particular, if $\phi\in\frakR(\Gamma)$ is a representation such that $\phi(\Gamma)$ is a semi-arithmetic Fuchsian group, then this is also true for every representation of $\Gamma$ that is equivalent to $\phi$ under the action of $\PGLTR$ on $\frakR(\Gamma)$ by conjugation. It thus makes sense to say that the point $[\phi]$ in the Teichm\"uller space $\T(\Gamma)$ is semi-arithmetic.

\begin{rmk}\label{18220.2}
	The Teichm\"uller space $\T(\Gamma)$ is known to be parametrised by finitely many trace functions (see, for example, \cite{Nakanishi98}, \cite{Nakanishi16}, \cite{Okumura90}). More precisely, let $\gamma\in \Gamma$ and define the \emph{trace function} $\Tr_\gamma : \T(\Gamma) \to \R$ as $\Tr_\gamma ([\phi]) = \Tr(\phi(\gamma))$. Then there exist $N>0$ and $\gamma_1,\dots,\gamma_N \in \Gamma$ such that $(\Tr_{\gamma_1},\dots,\Tr_{\gamma_N}) : \T(\Gamma) \to \R^N$ in an (real-analytic) embedding. It then follows from Definition $\ref{DDay19.2}$ that the semi-arithmetic points in $\T(\Gamma)$ are in one-to one correspondence with a subset of the $N$-tuples with algebraic integer coordinates. Therefore, we conclude that there are at most countably many semi-arithmetic Fuchsian groups.
\end{rmk}

We recall the following useful proposition concerning the trace set of a finitely generated Fuchsian group:

\begin{prop}[Lemma 2 in \cite{Takeuchi77}]\label{9320.1}
	Let $\Gamma$ be a Fuchsian group with generators $\gamma_1, \dots, \gamma_n$. For a subset of indices $\{i_1,\dots,i_k\} \subset \{1,\dots,n\}$, we denote $t_{i_1,\dots,i_k} = \Tr(\gamma_{i_1} \cdots \gamma_{i_k})$. Then
	
	\begin{align*}
	\Tr\Gamma \subset \Z\big[ \, t_{i_1,\dots,i_k} \mid \{i_1,\dots,i_k\} \subset \{1,\dots,n\} \, \big].
	\end{align*}
	
\end{prop}

\subsection{Congruence subgroups of semi-Arithmetic Fuchsian groups} \label{congruencesub}

Let $\Gamma < \PSLTR$ be a semi-arithmetic Fuchsian group with invariant trace field $K$ and let $\tilde{\Gamma}$ be the preimage of $\Gamma$ by the projection $P:\SLTR \to \PSLTR.$ By an equivalent definition of semi-arithmetic groups (see \cite[Proposition 1]{SW}), there exist a quaternion algebra $B$ over $K$, a maximal order $\mathfrak{O} < B$ and an injective morphism of $K$-algebras $\xi:B \to \mathrm{M}(2,\R)$ such that $\xi(\mathfrak{O}^1)< \SLTR$ is a group containing $\widetilde{\Gamma}$ or $\widetilde{\Gamma}^{(2)}$,  where $\mathfrak{O}^1$ denotes the group of elements in  $\mathfrak{O}$ with reduced norm 1.

For any ideal $\mathfrak{a} \subset \mathcal{O}_K$ we define $\mathfrak{a}\mathfrak{O}=\{\sum_j a_j \omega_j \mid  a_j \in \mathfrak{a} \mbox{ and } \omega_j  \in \mathfrak{O} \}$.  The \emph{congruence subgroup} of $\mathfrak{O}^1$ of level $\mathfrak{a}$ is given by
$$\mathfrak{O}^1(\mathfrak{a})= \{D \in \mathfrak{O}^1 \mid D-{\bf 1} \in \mathfrak{a} \mathfrak{O} \}.$$

We refer to \cite{K15} or \cite{KSV07} for a proof that $\mathfrak{O}^1(\mathfrak{a})$ is a normal subgroup of $\mathfrak{O}^1$ of finite index. The \emph{congruence subgroups} $\Gamma(\mathfrak{a})$ of level $\mathfrak{a}$ of $\Gamma$ are the projection on $\PSLTR$ of the intersection $\tilde{\Gamma} \cap \xi(\mathfrak{O}^1(\mathfrak{a}))$. Moreover, $\Gamma(\mathfrak{a})$ is a finite index subgroup of $\Gamma$. Indeed, consider the normal subgroup of finite index $\Gamma^{(2)} \unlhd \Gamma$, since $\Gamma^{(2)} < P(\xi(\mathfrak{O}^1))$ and $P(\xi(\mathfrak{O}^1(\mathfrak{a})))$ is normal of finite index in $P\xi(\mathfrak{O}^1)$, then $\Gamma^{(2)}(\mathfrak{a})$ is normal in $\Gamma^{(2)}$ and the natural map $\Gamma^{(2)}/\Gamma^{(2)}(\mathfrak{a}) \to P\xi\mathfrak{O}^1/P\xi\mathfrak{O}^1(\mathfrak{a})$ is injective. Hence, $\Gamma^{(2)}(\mathfrak{a})$ has finite index in $\Gamma^{(2)}$ and a fortiori has finite index in $\Gamma$. Thus, the relation $\Gamma^{(2)}(\mathfrak{a}) < \Gamma(\mathfrak{a}) < \Gamma$ implies that $[\Gamma:\Gamma(\mathfrak{a})]<\infty.$  

In \cite{K15} R. Kucharczyk obtained the following two important algebraic properties of congruence subgroups of semi-arithmetic Fuchsian groups.

\begin{theorem}[Theorem B and Proposition 4.5] \label{quotientcongruence}
Let $\Gamma$ be a semi-arithmetic Fuchsian group with invariant trace field $K$. There exists an infinite family $\mathcal{F}$ of prime ideals of $\mathcal{O}_K$ such that:
\begin{enumerate}[label=(\roman*)]
    \item $\Gamma^{(2)}/\Gamma^{(2)}(\mathfrak{p}) \mbox{ is isomorphic to } P\mathfrak{O}^1 /P\mathfrak{O}^1(\mathfrak{p})$ for every $\mathfrak{p} \in \mathcal{F}$;
    \item $ \Gamma^{(2)}/\Gamma^{(2)}(\mathfrak{p}) \mbox{ is isomorphic to } \SL(2,\mathcal{O}_K/\mathfrak{p})$ for every $\mathfrak{p} \in \mathcal{F}$.
\end{enumerate}
\end{theorem}

Since $\mathcal{O}_K/\mathfrak{p}$ is a finite field with $N(\mathfrak{p})$ elements, the following holds:

\begin{cor}\label{indexgrowth}
 Let $\Gamma$ be a semi-arithmetic Fuchsian group with invariant trace field $K$ and let $\mathcal{F}$ be the family of prime ideals of $\mathcal{O}_K$ given in the previous theorem. Then 
 $$ \frac{1}{4}N(\mathfrak{p})^3 \leq [\Gamma^{(2)}:\Gamma^{(2)}(\mathfrak{p})]
 \leq \frac{1}{2}N(\mathfrak{p})^3 $$ 
 for every $\mathfrak{p} \in \mathcal{F}.$
\end{cor}
\begin{proof}
Since $N(\mathfrak{p}) \geq 2$ for any $\mathfrak{p} \in \mathcal{F},$ the inequalities follow from the theorem and the elementary formula $|\PSL(2,\mathcal{O}_K/\mathfrak{p})|=\frac{1}{2}(N(\mathfrak{p})^2-1))N(\mathfrak{p})$.
\end{proof}

\subsection{Groups and geometry}\label{group}
For any finitely generated group $G$ with a set of generators $\Sigma$ we define a function $w_\Sigma:G \to \Z_{\ge 0}$, known as the \emph{word length} with respect to $\Sigma$, given by $w_\Sigma(g)=$ minimal length of a word in the alphabet $\Sigma \cup \Sigma\inv$ that represents $g$. Note that $w_\Sigma(g)=0$ if and only if $g=1$. Moreover, $d(g,h) \defeq \omega_\Sigma(gh\inv)$ is known to define a metric in $G$.

For a hyperbolic element $\gamma\in \PSLTR$ acting on $\Hy$, we define its \emph{translation length} $\ell(\gamma)$ as $\ell(\gamma) = \min \{d_{\Hy}(x,\gamma(x)) \mid x\in \Hy \}$. It is worth noting that $\ell(\gamma)$ is positive and is attained on any point on the \emph{axis} of $\gamma$, i.e., on the (unique) geodesic line invariant by $\gamma$.

The following is a useful lemma linking the geometry of a group to its action on a geometric space. It is a straightforward corollary of the Milnor-Schwarz Lemma.

\begin{lemma}\label{milnorlemma}
 Let $\Gamma < \PSLTR$ be a cocompact Fuchsian group with a finite set of generators $\Sigma$ and let $w_\Sigma(\gamma)$ denote the word length with respect to $\Sigma$ of any element $\gamma \in \Gamma$. There exist constants $C>0, c \in \R$ which depend only on $\Gamma$ and $\Sigma$ such that
   $$\ell(\gamma) \geq C w_\Sigma(\gamma) - c$$
for any hyperbolic element $\gamma \in \Gamma$. 
\end{lemma}
\begin{proof}
See Equation $5$ in the proof of the Proposition 10 in \cite{GL14} for the proof of this inequality for an specific set of generators. Since any two word lengths are Lipschitz equivalent, the result holds for any set of generators.
\end{proof}

\section{Geometry of semi-arithmetic surfaces}\label{densetoflengths}

\subsection{Construction}

In this subsection we prove that the set of numbers realised as the length of a closed geodesic in semi-arithmetic surfaces of some fixed genus is dense in the positive real numbers. More precisely, we describe a dense set $\mathcal{L}\subset [0,+\infty)$ such that, for any integer $g\geq2$ and $l\in \mathcal{L}$, there exists a semi-arithmetic surface $S$ of genus $g$ and a closed geodesic $\gamma$ in $S$ with length $\ell(\gamma) = l$.

\begin{lemma}\label{12919.3}
	Let $G$ be a 1-dimensional Lie group with finitely many connected components. If $H<G$ has rank $\geq 2$ then $H$ is dense in $G$
\end{lemma}

\begin{proof}
	Assume, without loss of generality, that $G$ is connected. If $G$ is non-compact then it is Lie group-isomorphic to $(\R, +)$ by an isomorphism, say, $f$. It is known that additive subgroups of $\R$ are either cyclic infinite or dense. Since $f(H)$ has rank $\geq 2$, it must be dense in $\R$ in which case $H$ is dense in $G$. If $G$ is compact, it will be isomorphic to $\Sph^1$, where a similar dichotomy holds.
\end{proof}

Lemma \ref{12919.3} together with the Dirichlet's Unit Theorem (Theorem \ref{dirichlet}) give us the following corollary. 

\begin{cor}\label{14919.1}
	Let $K$ be a totally real number field such that $[K:\Q]\geq 3$. Then $\mathcal{O}_K^\times$ is dense in $\R$.
\end{cor}

Next, we list some formulae in classical hyperbolic geometry that will be used later on. For a systematic treatment of the subject, see, for example, \cite{Beardon12}, \cite{Buser92} or \cite{Fenchel89}.

\begin{lemma}[{\cite[Theorems 7.11.1 and 7.17.1]{Beardon12}}] Consider a geodesic quadrilateral, known as \emph{trirectangle}, with three right angles and one acute angle $\varphi$, with side lengths indicated as in Figure \ref{8320.1}. Then the following holds:
	
	\begin{align}\label{7320.1}
    &\cos \varphi = \sinh a \, \sinh b \, ;\\
	&\cosh d = \cosh a \, \cosh b. \label{101019.6}
	\end{align}
	
Note that equation \eqref{101019.6} is but the hyperbolic Pythagoras' Theorem.
\end{lemma}

\begin{figure}[h]
	\centering
	\includegraphics[scale=0.1]{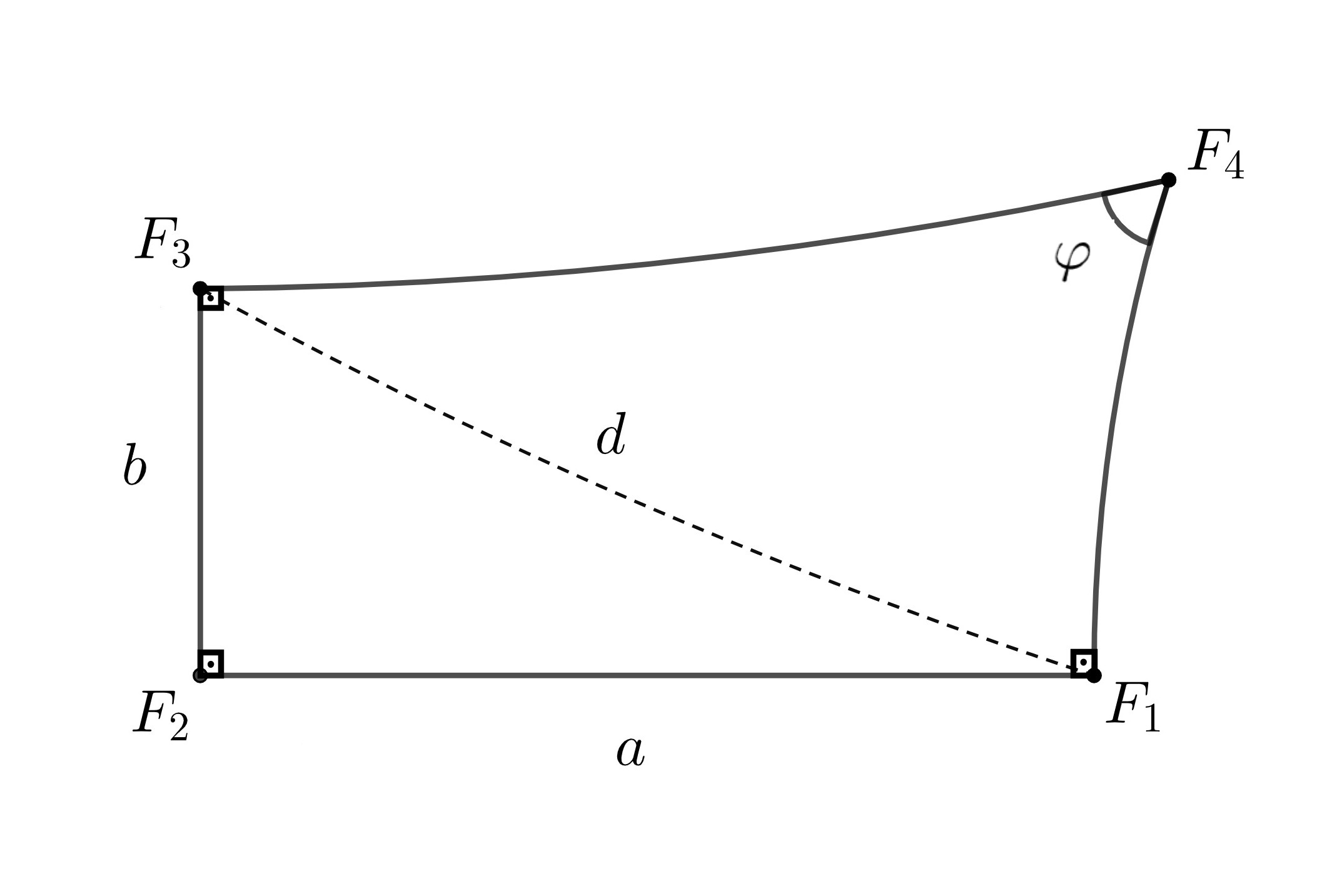}
	\caption{Trirectangle}\label{8320.1}
\end{figure}

\begin{lemma}[{\cite[Theorem 7.19.2]{Beardon12}}]\label{101019.3}
	For any three positive real numbers $a_1,a_2,a_3$, there exists a convex right-angled hexagon with three non-adjacent sides of length $a_1,a_2,a_3$. 
	
	Moreover, this hexagon is unique (up to isometry). Indeed, if the side of length $a_1$ is opposed by a side of length, say, $b_1$, then the following equation holds: 
	\begin{align*}
	\cosh b_1 \, \sinh a_2 \, \sinh a_3 = \cosh a_1 + \cosh a_2 \, \cosh a_3.
	\end{align*}
	In other words, the lengths of all sides of the hexagon are determined by the lengths of three non-adjacent sides. 
\end{lemma}

Finally, we recall the following relation between the displacement of a hyperbolic element and its trace.

\begin{prop}\label{31120.1}
	Let $\gamma \in \PSLTR$ be a hyperbolic element with translation length $\ell(\gamma)$. Then the following relation holds:
	
	\begin{align*}
	\Tr\gamma = 2\cosh \frac{\ell(\gamma)}{2}.
	\end{align*} 
\end{prop}

\begin{cor}[cf. {\cite[Theorem 7.38.2 ]{Beardon12}}] \label{101019.4}
	If $A,B\in\PSLTR$ are half-turns around $p_1$ and $p_2$, then
	\begin{align*}
	\Tr AB = 2\cosh d_\Hy(p_1,p_2),
	\end{align*}
	where $d_\Hy$ is the hyperbolic distance.
\end{cor}

Consider the abstract group 
\begin{align}\label{8320.2}
\Lambda  =	\langle s_1,\dots, s_4 \mid s_1^2 = s_2^2 = s_3^2 = s_4^3 = s_1\cdots s_4 = 1\rangle.
\end{align}

Let $\mathcal{Q}$ be a trirectangle with acute angle measuring $\pi/3$ and with vertices labeled $F_1,\dots,F_4,$ as shown in Figure \ref{8320.1}. Let $\sigma_1,\sigma_2, \sigma_3, \sigma_4$ denote the reflections across the geodesic lines supporting, respectively, the sides $F_1F_2, \ F_2F_3, \ F_3F_4, \ F_4F_1$, and let $\mathrm{R}(\mathcal{Q})$ be the group of isometries of $\Hy$ generated by these reflections. The index $2$ subgroup of orientation-preserving isometries is known to be a Fuchsian group with presentation given by \eqref{8320.2}. Moreover, the product of any two reflections across adjacent sides of $\cQ$ gives an elliptic element in $\PSLTR$ fixing the intersection point between the respective sides. In particular, the isometries defined by $S_j = \sigma_j\sigma_{j+1}$ for $j=1,2,3,4$ (where indices are taken modulo 4) 
are elliptic isometries fixing the vertex $F_j$, and the map $\rho_{\mathcal{Q}}:\Lambda \rightarrow \PSLTR$ given by $\rho_{\mathcal{Q}}(s_i)=S_i$, $i=1,2,3,4$, induces an injective homomorphism onto the Fuchsian group in question.

Consider the right-angled hexagon $\cH$ that is tiled by six copies of one such trirectangle, as in Figure \ref{8320.3}. One may verify that the isometries
\begin{align*}
	\sigma_1 \, , \, \sigma_2 \, , \, \sigma_3\sigma_1\sigma_3 \, , \, (\sigma_4\sigma_3\sigma_4)\sigma_2(\sigma_4\sigma_3\sigma_4) \, , \, \sigma_4(\sigma_3\sigma_1\sigma_3)\sigma_4 \, , \, \sigma_4\sigma_2\sigma_4
\end{align*}
are reflections across the geodesic lines supporting the respective sides of the hexagon. It follows that the group $\mathrm{R}(\cH)$, generated by reflections across the sides of $\cH$, is a subgroup of $\mathrm{R}(\cQ)$. Let $C_j$ denote the half-turn around the vertex $E_j$ of $\cH$, i.e., the product of the reflections across the sides of $\cH$ intersecting at $E_j$. Let $\Gamma$ be the abstract group with presentation
\begin{align*}
\Gamma = \langle c_1,\dots, c_6 \mid c_1^2 = \dots = c_6^2 = c_1\cdots c_6 = 1\rangle.
\end{align*}
In the same spirit as before, we define a map $\rho_{\cH}$ taking $c_j$ to $C_j$ and obtain an injective representation of $\Gamma$ into $\PSLTR$ such that 
\begin{align}\label{8320.4}
\rho_{\cH}(\Gamma) < \rho_{\cQ}(\Lambda).
\end{align}
Furthermore, we observe that $\rho_{\cH}(\Gamma)$ has index $6$ in $\rho_{\cQ}(\Lambda)$.

\begin{figure}[h]
	\centering
	\includegraphics[scale=0.4]{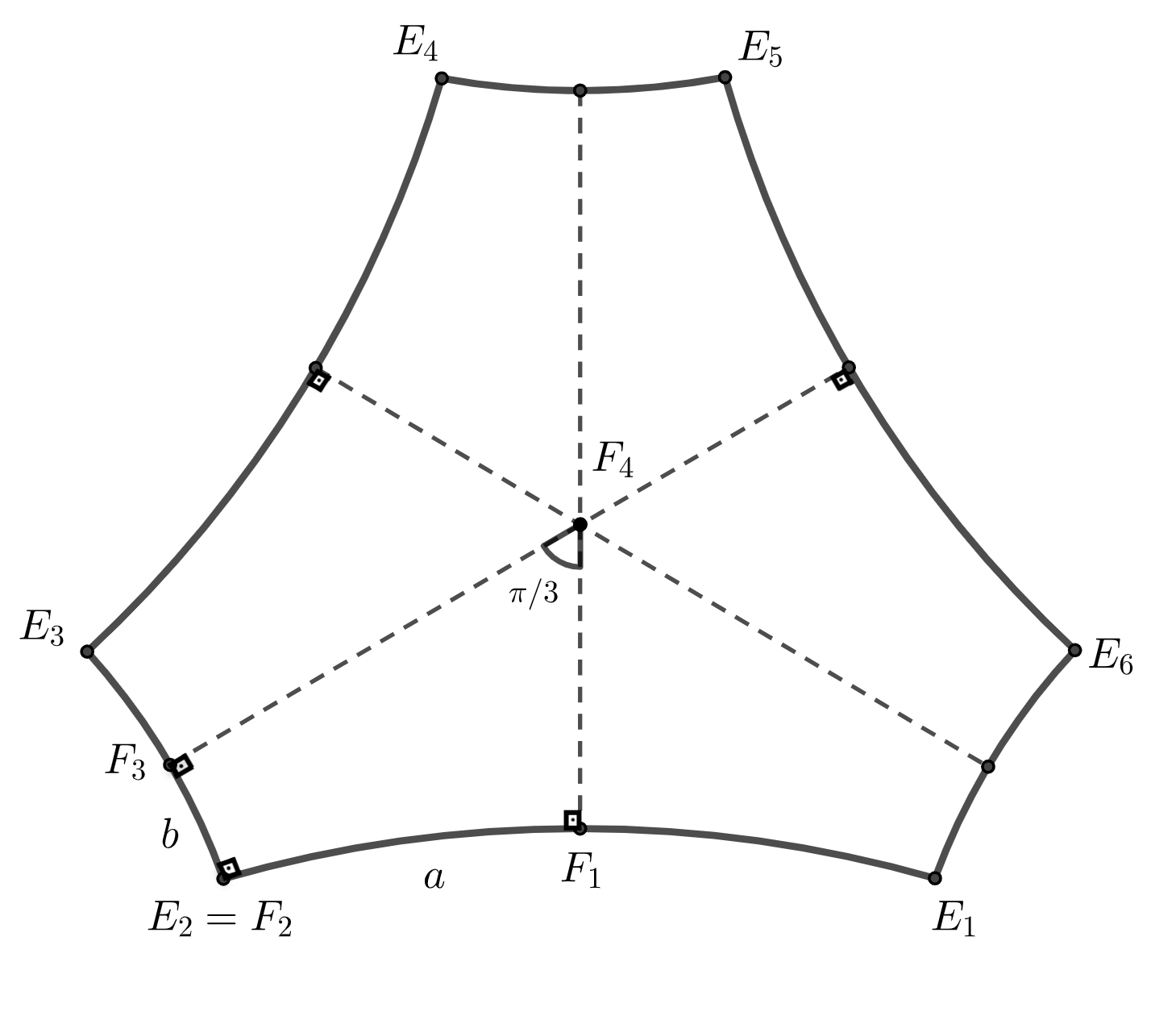}
	\caption{Right-angled hexagon} \label{8320.3}
\end{figure}

\begin{definition}
	We say that a real number $t\geq0$ is \emph{realised} by the Fuchsian group $\Gamma$ if there exists some element $\gamma \in \Gamma$ such that $\Tr\gamma = t$. Similarly, we say that $l>0$ is realised by the hyperbolic surface $S$ if there exists some closed geodesic in $S$ of length $l$. Note that $S=\Hy/\Gamma$ realises $l>0$ if and only if $\Gamma$ realises $2\cosh(l/2) >2$. Indeed, recall that a closed geodesic in $S$ is the projection of the axis of a hyperbolic element of $\Gamma$ with translation length equal to the length this geodesic, then use Proposition \ref{31120.1}.
\end{definition}

\begin{theorem}\label{251019.1}
	The set $\cT$, of all real numbers that are realised by a semi-arithmetic Fuchsian group of signature $(0;2,2,2,2,2,2)$, is dense in the interval $[2,+\infty)$.
\end{theorem}

\begin{proof}
	Consider a trirectangle $\mathcal{Q}_a$ with angle $\pi/3$ at the vertex $F_4$, and side $F_1F_2$ of length $a$, as shown in Figure \ref{8320.1}. Let $\cH_{2a}$ denote the corresponding right-angled hexagon tiled by $\cQ_a$, as constructed above, and note that the side $E_1E_2$ has length $2a$.
	
	By Theorem 2.3.1 in \cite{Buser92}, the length $a$ determines the other three sides of the trirectangle. Note that, by an argument of continuity, $a$ can be arbitrarily chosen while keeping the acute angle with a fixed measure of $\pi/3$. We will select this length in such a way that the resulting group $\rho_{\mathcal{Q}_a}(\Lambda) < \PSLTR$ will be semi-arithmetic and, as a consequence, so will be the finite index subgroup $\rho_{\cH_{2a}}(\Gamma)$.
	
	Let $K$ be any totally real number field with $[K:\Q]\geq 3$ and let $V = \{v\in \R^+ \mid \sinh v \in \rO_K^\times \}$. So, in particular, for every $v\in V$, $\sinh v$ is a totally real unit in the ring of integers of $K$. It follows from Corollary \ref{14919.1} that $V$ is dense in $\left[0,\infty\right)$. Choose $a \in V$.
	
	We first observe that $\cosh^2 a = \sinh^2 a +1$ is a totally positive algebraic integer. It then follows that $\cosh a$ is a totally real algebraic integer (see Lemma \ref{10320.1}).
	
	Equation \eqref{7320.1} gives that $\sinh b = \frac{(\sinh a)\inv}{2}$ and then
	\begin{align*}
	\cosh^2b= \sinh^2 b + 1 = \frac{(\sinh a)^{-2} + 4}{4}.
	\end{align*}
	By the same reasoning as before, we obtain $2\cosh b = \sqrt{(\sinh a)^{-2} + 4}$ is a totally real algebraic integer.
	Now, by the Pythagoras' Theorem, we have that
	\begin{align*}
	2 \cosh d = \cosh a (2\cosh b)
	\end{align*}
	is also a totally real algebraic integer 
	
	Thus, by Corollary \ref{101019.4}, the traces $\Tr(S_1S_2), \Tr(S_2S_3),  \Tr(S_1S_3)$ are all totally real algebraic integers. Note also that the order 2 elements $S_1,S_2,S_3$ have trace 0, and that the order three element $S_4$ has trace 1 so, in view of Proposition \ref{9320.1} and \eqref{8320.4}:
	
	\begin{align*}
	\Tr(\rho_{\cH_{2a}}(\Gamma)) \subset \Tr(\rho_{\mathcal{Q}_a}(\Lambda)) \subset \Z[\Tr(S_1S_2), \Tr(S_2S_3),  \Tr(S_1S_3)].
	\end{align*}
	
	In particular, the invariant trace-field of $\rho_{\mathcal{H}_{2a}}(\Gamma)$ is a subfield of
	\begin{align*}
	\Q(\Tr(S_1S_2), \Tr(S_2S_3),  \Tr(S_1S_3)),
	\end{align*}
	which is, by construction, totally real. We therefore conclude that the Fuchsian group $\rho_{\mathcal{H}_{2a}}(\Gamma)$ is semi-arithmetic. Since $a$ was arbitrarily chosen among a dense subset $V$ of $[0,+\infty)$, the group $\rho_{\mathcal{H}_{2a}}(\Gamma)$ realises $t=2\cosh 2a = 2 + 4\sinh^2a$ as $\Tr\rho_{\cH_{2a}}(c_1c_2)$ for any $t = t(a)$ in $\cT \defeq 2\cosh 2V$, the latter being dense in $[2,+\infty)$.
\end{proof}

A \emph{surface-kernel epimorphism} is an epimorphism $\theta\! : \Gamma \to G$ with torsion-free kernel. This is the case, for example, if $\theta$ preserves the order of the torsion elements of $\Gamma$. Note that it is sufficient to check if the orders of the generators in \eqref{201019.1} are preserved.

\begin{theorem} \label{mainthm}
	Let $\mathcal{L}_g$ be the set of real numbers that are realised by a semi-arithmetic surface of genus $g$. Then $\displaystyle\bigcap_{g\geq 2} \mathcal{L}_g$ is dense in $[0,+\infty)$. 
\end{theorem}

\begin{proof}
Let $\cT$ be the set obtained in Theorem \ref{251019.1}. For each $t \in \cT$ define $\Gamma_t=\rho_t(\Gamma):=\rho_{\mathcal{H}_{2a}}(\Gamma)$ where $a$ is such that $2\cosh 2a = t$, as in the construction of the previous theorem. From what it was proved, it follows that the Fuchsian group $\Gamma_t$ is a semi-arithmetic group of signature  $(0;2,2,2,2,2,2)$ with $\Tr \rho_t(c_1c_2) = t.$\\

\textit{Case 1}: $g=2$

Let $\theta\!:\Gamma \to \Z/2\Z$ be a homomorphism defined by $c_i \mapsto \overline{1}$, $i=1,\dots,6$. It is immediate that $\theta$ is a surface-kernel epimorphism, since it preserves the order of the generators of $\Gamma$.  It follows that $\mathcal{K}:=\ker\theta$ is a torsion-free index 2 subgroup of $\Gamma$ and, as such, it must be a surface group of genus $g$. From \eqref{241019.1} and the Riemann-Hurwitz formula, we compute that
\begin{align*}
2\pi(2g-2) = \area(\rho_t(\mathcal{K})\backslash\Hy) = [\Gamma_t : \rho_t(\mathcal{K})]\cdot\area(\Gamma_t\backslash\Hy) = 2 \cdot 2\pi ,
\end{align*}
and so $g=2$. Moreover, $\theta(c_1c_2) = \overline{1}+\overline{1}= \overline{0}$ and thus $c_1c_2 \in \mathcal{K}$.

The quotient, $\rho_t(\mathcal{K})\backslash\Hy$, is a closed hyperbolic surface of genus 2.  The axis of the hyperbolic element $\rho_t(c_1c_2)$ projects onto a closed geodesic $\gamma_t$ in $\rho_t(\mathcal{K})\backslash\Hy$ that satisfies $2\cosh(\ell(\gamma_t)/2) = t$. If we let $t$ vary in the set $\cT$, then $\ell(\gamma_t)= 2\cosh\inv(t/2)$ covers a dense subset of $[0,+\infty)$.\\

\textit{Case 2}: $g>2$

For larger genera we proceed similarly. Using the Reidemeister process (see Appendix \ref{procedemeister} for the details) we can find an isomorphism $\Phi$ between $\langle x,y,x',y' \mid [x,y][x',y'] \rangle$ and $\mathcal{K}$ such that $\Phi(y) = c_1c_2$. We will mildly abuse notation and say that 

\[\mathcal{K}=\langle x,y,x',y' \mid [x,y][x',y'] \rangle \mbox{ and } y=c_1c_2.\]

For each $n > 1$ we define the homomorphism $\eta_n\!:\mathcal{K} \to \Z/n\Z$ given by 
\begin{align*}
&\eta_n(x) = \eta_n(x')= \eta_n(y') = \bar{1},\\
&\eta_n(y) = \bar{0}.
\end{align*}

Therefore for all $n>1$ we have: $\eta_n$ is an epimorphism, $y \in \mathcal{K}_n:=\ker\eta_n$ and $\mathcal{K}_n$ is a surface group of genus $n+1$, since $[\mathcal{K} : \mathcal{K}_n]=n$. Moreover, the geodesic $\gamma_t$ obtained in the case 1, lifts as a closed geodesic of the same length for all coverings $\rho_t(\mathcal{K}_n) \backslash \Hy$ of $\rho_t(\mathcal{K})\backslash\Hy$. Thus, the same conclusion as in the case $g=2$ holds for any genus $g>2$.
\end{proof}

\begin{repthrm}[\ref{dense:thm}]
	For any $g \ge 2$ there exists a length function $\ell:\rT_g \to \mathbb{R}$ such that 
	$$\{ \ell(S) \mid S \in \rT_g \mbox{ is semi-arithmetic} \}$$
	is dense on the set of positive real numbers.
\end{repthrm}

\begin{proof}
	This follows straight from (the proof of) Theorem \ref{mainthm} (see Remark \ref{15320.1}). Indeed, the axis of the hyperbolic element $\rho_t(c_1c_2)$ projects onto a closed geodesic in $\rho_t(\cK_n)\backslash\Hy$, whose free homotopy class induces the length function $\ell$ with the desired properties.
\end{proof}

\subsection{Applications}\label{19420.4}

Our first application will be Theorem \ref{cor:densityofsys}, whose proof uses the fact that $\mathcal{K}$ is a \emph{limit group}. In fact, in Lemma \ref{hom} we prove a stronger result. Similar ideas were used in \cite[Theorem 5.4]{D17}. Before we proceed, however, we must introduce some notation.

Let $\pi_1(S_2)$ be the fundamental group of an orientable compact surface of genus $2$ and let $\F_2$ be the free group of rank 2 with the set of generators $A=\{x, y\}$ . If we consider a new copy $\F_2'$ of $\F_2$ with the set of generators $A'=\{ x', y' \}$ and let $u=[x,y] \in \F_2, v=[y',x'] \in \F_2'$, it is well known that $\pi_1(S_2)$ is isomorphic to the quotient $\F_2 \ast \F_2' / \langle\langle u*v^{-1} \rangle\rangle $, where $\F_2 \ast \F_2'$ denotes the free product of $\mathbb{F}_2$ and $\mathbb{F}_2'$.

Considering the natural monomorphisms $\iota: \F_2 \rightarrow \pi_1(S_2)$ and $\iota': \F_2' \rightarrow \pi_1(S_2)$  we can identify $\F_2, \F_2'$ as subgroups of $\pi_1(S_2)$. With this identification in mind, take $B=A \cup A'$ as a set of generators of $\pi_1(S_2)$.

\begin{lemma}[{\cite[Proposition 4.3]{D17}}] \label{hom}
	There exists a constant $\varepsilon > 0$ such that for every positive integer $k$ we can find an epimorphism $\psi_k: \pi_1(S_2) \rightarrow \F_2$ with the following properties:
	\begin{enumerate}
		\item \emph{Left inverse of $\iota$:} \label{29820.3}
		\begin{align*}
			\psi_k(\iota(z))=z \quad\text{ for every z } \in \F_2;
		\end{align*}
		\item \emph{$k\varepsilon$-Lipschitz:} \label{29820.2}
		\begin{align*}
			w_A(\psi_k(\eta)) \leq \varepsilon k w_B(\eta) \quad \text{for all } \eta \in \pi_1(S_2).
		\end{align*}
	\end{enumerate}
	\vspace{2mm}
	Moreover, $\psi_k$ satisfies the more technical property: 
	\vspace{2mm}
	\begin{enumerate}[resume]
		\item For any $\gamma \in \pi_1(S_2)$ with $1 \leq w_B(\gamma)\leq k$ we have $\psi_k(\gamma) \neq 1$ and, furthermore, if $\gamma \notin \langle x \rangle$ (respectively $\gamma \notin \langle y \rangle$), then $\psi_k(\gamma) \notin \langle x \rangle$ (respectively $\psi_k(\gamma) \notin \langle y \rangle$).\label{29820.1}
	\end{enumerate}
	
	\begin{rmk}
		Although item \eqref{29820.1} is not entirely contained in the statements of Proposition 4.3 in \cite{D17}, it follows from the construction of the map $\psi_k$ therein.
	\end{rmk}

\end{lemma}

Let $X$ be a non-empty set and let $f: X \to X$ be a permutation of $X$. We define its support by $\mathrm{supp}(f)=\{x \in X \mid f(x) \neq x \}$. In this case, for any subset $Y$ of $X$ containing $\mathrm{supp}(f),$ the restriction of $f$ to $Y$ is a permutation of $Y$. Finally, given non-empty subsets $U,V \subset \R$ we say that $U > V$ if $u>v$ for any $u \in U$ and $v \in V.$
%
%
%

\begin{lemma}\label{action}
 For any integer $m \ge 1$, there exist a finite set $X_m$, a point $p\in X_m$, and an action $\F_2 \curvearrowright X_m$ satisfying:
 \begin{enumerate}
  \item $y \cdot p = p$;
  \item If $z \cdot p = p$ and $w_A(z) \leq m$, then $z=y^l$ for some $l$.
 \end{enumerate}

\end{lemma}
\begin{proof}
    We first prove the following claim:

    {\noindent\bf Claim 1:} Let $k \geq 1$ be an integer. For any finite set $Y \subset \N$ there exists a permutation $\tau:\N \to \N$ such that $\mathrm{supp}(\tau)=Y \cup Z$ where $Z=\{ \tau^l(x) \mid x \in Y \mbox{ and } 0<|l|\leq k\}$ and $Z > Y$.

\medskip

Indeed, we can suppose that $Y=\{1,\ldots,a\}$ for some $a \geq 1$. Take $\tau$ as the product $\tau_1\cdots\tau_a$ of disjoint cycles given by $\tau_i=(i \,\, i+a  \,\, i+2a \,\, \ldots \,\, i+ka).$  Note that for any $1\leq i \leq a$ and $l$ with $0<|l|\leq k,$ we have 
\[\tau^l(i)=\tau_i^l(i) \in \{i+|l|a,i+(k-|l|+1)a\} \Rightarrow \tau^l(i)>a.\]

Now we can return to the proof of the lemma. We apply the claim above for the set $Y_1=\{1,\ldots,m\}$ with $k=m$ in order to get a permutation $\zeta_1:\N \to \N$ such that $\mathrm{supp}(\zeta_1)=Y_1 \cup Y_2$, where $Y_2=\{\zeta_1^l(t) \mid t \in Y_1 \mbox{ and } 0<|l|\leq m\}$ satisfies $Y_2>Y_1.$  If we apply the claim inductively, with $k=m$ and $Y=Y_q$, we obtain for every $q \geq 1$ a permutation $\zeta_q:\N \to \N$ with support equal to the union of $Y_q$ and $Y_{q+1}:= \{\zeta_q^l(t) \mid t \in Y_q \mbox{ and } 0<|l|\leq m\}$ and such that $Y_{q+1}>Y_q.$

Consider the set
\begin{align*}
	X_m=\{0\} \cup \mathrm{supp}(\zeta_1) \cup \mathrm{supp}(\zeta_2) \cup \cdots \cup \mathrm{supp}(\zeta_{2m-1}).
\end{align*}

By construction, $X_m$ is finite and we can see any $\zeta_i$ as a permutation of $X_m$. If we define $Y_{0}=\{0\}$ and $\zeta_0=(0 \,\, 1 \,\, \cdots \,\, m)$ in $X_m$, we may write that $\mathrm{supp}(\zeta_i)=Y_{i}\cup Y_{i+1}$ and $Y_{i+1}>Y_i$ for any $i=0,\ldots,2m-1$. In particular,
\begin{align*}
	X_m = Y_0\cup Y_1 \cup \cdots \cup Y_{2m},
\end{align*}
and, for $i\neq j$ with the same parity, we have that $\mathrm{supp}(\zeta_i) \cap \mathrm{supp}(\zeta_j)=\emptyset.$ We may therefore define the action of $\F_2$ on $X_m$ in the following way: for any $\alpha, \beta \in X_m,$ 
\begin{align*}
 x \cdot \alpha & = \zeta_{2p} \cdot \alpha  \mbox{, if } \alpha \in \mathrm{supp}(\zeta_{2p}) \mbox{ for some } p, \mbox{ otherwise } x \mbox{ fixes } \alpha. \\
 y \cdot \beta &  =  \zeta_{2q+1} \cdot \beta \mbox{, if } \beta \in \mathrm{supp}(\zeta_{2q+1}) \mbox{ for some } q, \mbox{ otherwise } y \mbox{ fixes } \beta.
 \end{align*} 

It follows from definition that $x$ only fixes the elements of $Y_m$, while $y$ fixes the elements of $Y_0$, i.e., $y \cdot 0 = 0$. Suppose $\omega_A(z) \leq m$ and $z \notin \langle y \rangle$. We may write 
\begin{align*}
z = y^{i_0}\,x^{j_1}\,\cdots \,x^{j_\kappa}\,y^{i_\kappa} \mbox{ with } j_1i_1j_2 \cdots i_{k-1}j_\kappa \neq 0.    
\end{align*}
where \[\kappa \leq |i_0| + \sum_{r=1}^\kappa (|i_r|+|j_r|) = \omega_A(z) \leq m.\]

We need to check that $z \cdot 0 \neq 0.$ Indeed,
we note first that  $x^{j_\kappa} \cdot (y^{i_k} \cdot 0) \in Y_1$ since $|j_k| \leq m$. By the construction and the fact that each $|i_l|,|j_l| \leq m,$ it follows that the subword $z_\sigma=x^{j_\sigma} \,\cdots \,x^{j_\kappa}\,y^{i_\kappa}$ satisfies $z_\sigma \cdot 0  \in Y_{1+2(\kappa-\sigma)}$. Since $1 \leq \kappa \leq m$, we get $z_1 \cdot 0 \in \displaystyle\cup_{j=1}^mY_{2j-1}$ and finally $z \cdot 0 \in \displaystyle\cup_{i=1}^{2m} Y_i = X_m \setminus \{0\}.$

\end{proof}

Now we are ready for the proof of:

\begin{repthrm}[\ref{cor:densityofsys}]
	The set $\{ \sys(S) \mid S \mbox{ \emph{is a closed semi-arithmetic Riemann surface}} \}$ is dense in the positive real numbers.
\end{repthrm}

\begin{proof}

Let $t \in \cT$ be fixed, where $\cT$ is the set given by Theorem \ref{251019.1}. We consider the Fuchsian group $\rho_t(\mathcal{K})$ with set of generators $\Sigma_t:=\rho_t(\Sigma)$, where $\Sigma=\{x,y,x',y'\}$ is the standard set of generators of $\mathcal{K}$, constructed in the proof of Theorem \ref{dense:thm}. By Lemma \ref{milnorlemma}, there exist constants $C(t),c(t)$ with $C(t)>0$ such that, for any hyperbolic element $\gamma \in \rho_t(\mathcal{K})$ we have
\begin{align}\label{D100320.1}
 \ell(\gamma) \geq C(t)w_{\Sigma_t}(\gamma)-c(t),   
\end{align}
where $w_{\Sigma_t}(\gamma)$ denotes the word length of $\gamma$ with respect to $\Sigma_t$.

Now define 
$\nu(t)=\Bigl\lceil \frac{|c(t)|+2\cosh\inv(t/2)}{C(t)} \Bigr\rceil$ and $\mu(t) = \lceil \varepsilon \nu(t)^2 \rceil$, where $\varepsilon > 0$ is the constant given by Lemma \ref{hom} and $\lceil x \rceil = \min\{ n \geq 1 \mid x \leq n \}$ for any $x>0$.

By Lemma \ref{action}, there exist a finite set $X_{\mu(t)}$ and an action $\F_2 \curvearrowright X_{\mu(t)}$ such that, for some $p \in X_{\mu(t)}$, its isotropy group $H_t=\{\zeta \in \F_2 \mid \zeta \cdot p = p \}$ contains $y$ and satisfies 
\begin{equation}\label{D23820.1}
  z \in H_t, \,\, w_A(z) \leq \mu(t) \Rightarrow z \in \langle y \rangle.
\end{equation}

By Lemma \ref{hom}, there exists an epimorphism $\psi:=\psi_{\nu(t)}:\mathcal{K} \rightarrow \F_2$ such that $\psi(\eta) \neq 1$ if $1<w_{\Sigma_t}(\eta) \leq \nu(t)$ and, for all $\xi \in \mathcal{K}$, $\omega_A(\psi(\xi)) \leq \varepsilon\, \nu(t) \,  w_{\Sigma_t}(\xi)$. Now we consider the finite index subgroup $\Lambda_t=\psi\inv(H_t) < \mathcal{K}$.

\medskip

{\noindent\bf Claim 1:} Any hyperbolic element $\beta \in \rho_t(\Lambda_t) \setminus \langle \rho_t(y) \rangle $ satisfies

\[w_{\Sigma_t}(\beta) > \nu(t).\]

\medskip

We argue by contradiction. Suppose there exists $\beta=\rho_t(\eta)$, for some (unique) $\eta \in \Lambda_t \setminus \langle y \rangle$, with  $w_\Sigma(\eta) = w_{\Sigma_t}(\beta)\leq\nu(t)$. By Lemma \ref{hom} (item \eqref{29820.1}) we have $\psi(\eta) \neq 1$ and $\psi(\eta) \in H_t \setminus \langle y \rangle $. Then, it follows from our choices of $\mu(t)$ and $\nu(t)$, together with \eqref{D23820.1} and property $(3)$ of $\psi$ (in Lemma \ref{hom}) that

\[\varepsilon\nu(t)^2 \leq \mu(t) < w_A(\psi(\eta)) \leq  \varepsilon \nu(t)  w_\Sigma(\eta)=\varepsilon \nu(t) w_{\Sigma_t}(\beta)\leq \varepsilon \nu(t)^2.\]

This contradiction implies that \[w_{\Sigma_t}(\beta)> \nu(t)\geq \frac{c(t)+2\cosh \inv (\frac{t}{2})}{C(t)}. \]

\bigskip

Now we conclude the argument by showing that the systole $s(t)$ of the closed surface $\rho_t(\Lambda_t) \backslash \Hy$ is realised by the hyperbolic element $\rho_t(y)$, i.e. $s(t)=2\cosh \inv (\frac{t}{2}).$

Indeed, since $\psi(y)=y$ (item $(1)$ of Lemma \ref{hom}), then $\rho_t(y) \in \rho_t(\Lambda_t)$ corresponds to a closed geodesic of length $2\cosh \inv (\frac{t}{2})$, in particular $s(t) \leq 2\cosh \inv (\frac{t}{2})$. On the other hand, let $\alpha$ be a hyperbolic element in $\rho_t(\mathcal{K})$ corresponding to the systole of the covering $\rho_t(\Lambda_t) \backslash \Hy$. If $\alpha \in \langle \rho_t(y) \rangle$ then $s(t) \geq 2\cosh \inv (\frac{t}{2})$ and we are done.
Actually, $\alpha$ must be in the cyclic group $\langle \rho_t(y) \rangle$, for if it is not, Claim $1$ would imply that $w_{\Sigma_t}(\alpha) > \nu(t)$ and \eqref{D100320.1}, together with our choice for $\nu(t)$, would then give 
	\begin{align*}
		s(t)=\ell(\alpha) \geq C(t)w_{\Sigma_t}(\alpha)-c(t) > C(t)\nu(t)-c(t) \geq 2\cosh\inv(t/2),
	\end{align*}
contrary to what we have established.

Since  $t \in \cT$ was arbitrarily chosen and $\cT$ is dense in the interval $[2,\infty)$, we conclude that $\{s(t) \mid t \in \cT\}$ is a dense subset of the positive real numbers and hence, so is the set $\{\sys(S) \mid S \mbox{ \emph{is a closed semi-arithmetic Riemann surface}} \}$. 
\end{proof}

Next, we show how to realise infinitely many number fields as the invariant trace field of a semi-arithmetic Fuchsian group with a fixed genus $g\geq 2$. The idea is essentially contained in Theorem \ref{251019.1} and Theorem \ref{mainthm}.

\begin{theorem}\label{20420.1}
Every totally real number field of prime degree at least 3 is realised as the invariant trace field of a genus $g$ Riemann surface with integral traces.
\end{theorem}

\begin{proof}
	Let $K$ be a totally real number field of prime degree $p\geq 3$ and $a$ a positive real number such that $\sinh a \in \cO_K^\times$. For each $g\geq 2$, we can find a semi-arithmetic genus $g$ Riemann surface whose uniformising Fuchsian group $\Delta$ realises $2\cosh 2a = 2 + 4\sinh^2 a$, i.e., there exists some $\gamma \in \Delta$ with $\Tr\gamma = 2 + 4\sinh^2 a$. Note that $a$ may be chosen in such a way that $\Tr^2 \gamma$ is not a rational number so that $k\Delta$ is strictly larger than $\Q$. Since $[k\Delta : \Q]$ divides the prime number $[K:\Q]$, we conclude that  $k\Delta = K$.
\end{proof}

\begin{rmk}
	In particular, there are semi-arithmetic Riemann surfaces of genus $g$ with invariant trace fields of arbitrarily large degree.
\end{rmk}

This gives a negative answer to a conjecture made by B. Jeon (see \cite[Conjectrue 2]{Jeon19}):

\begin{conjecture*}[Jeon]
	For each $g\geq 2$ there exists only a finite number of real number fields and quaternion algebras that are realised as the invariant trace field and invariant quaternion algebra of a hyperbolic structure on $S_g$ with integral traces.
\end{conjecture*}

\section{Semi-arithmetic surfaces with logarithmic systolic growth } \label{18220.1}

In this section, we construct sequences of surfaces of logarithmic systolic growth with multiplicative constant depending only on arithmetic and geometric data. 

\begin{lemma}\label{areagrowth}
Let $S=\Gamma \backslash \Hy$ be a semi-arithmetic Riemann surface and let $K$ be its invariant trace field. Consider the family $\mathcal{F}$ of prime ideals of $\mathcal{O}_K$ given in Theorem \ref{quotientcongruence}. For each $\mathfrak{p} \in \mathcal{F}$ let 
$S_{\mathfrak{p}}=\Gamma(\mathfrak{p}) \backslash \Hy$ be the congruence covering of $S$ of level $\mathfrak{p}$. There exist constants $l=l(S),L=L(S)$ such that
\[lN(\mathfrak{p})^3 \leq \area(S_{\mathfrak{p}}) \leq L N(\mathfrak{p})^3 \]
for every $\mathfrak{p} \in \mathcal{F}$.
\end{lemma}

\begin{proof}
Since $S_{\mathfrak{p}}$ is a covering of degree $[\Gamma:\Gamma(\mathfrak{p})]$, we have $\area(S_{\mathfrak{p}})=\area(S)[\Gamma:\Gamma(\mathfrak{p})]$ for any $\mathfrak{p} \in \mathcal{F}$. The equality $$[\Gamma:\Gamma^{(2)}][\Gamma^{(2)}:\Gamma^{(2)}(\mathfrak{p})]=[\Gamma:\Gamma(\mathfrak{p})][\Gamma(\mathfrak{p}):\Gamma^{(2)}(\mathfrak{p})] $$
and the inequality $[\Gamma(\mathfrak{p}):\Gamma^{(2)}(\mathfrak{p})] \leq [\Gamma:\Gamma^{(2)}]$ together with the Corollary \ref{indexgrowth} imply that
\[\frac{1}{4}N(\mathfrak{p})^3 \leq  [\Gamma^{(2)}:\Gamma^{(2)}(\mathfrak{p})] \leq [\Gamma:\Gamma(\mathfrak{p})] \leq [\Gamma:\Gamma^{(2)}][\Gamma^{(2)}:\Gamma^{(2)}(\mathfrak{p})] \leq \frac{1}{2}[\Gamma:\Gamma^{(2)}]  N(\mathfrak{p})^3.\]

The constants $l(S)=\frac{1}{4}\area(S)$ and $L(S)=\frac{1}{2}[\Gamma:\Gamma^{(2)}]\area(S)$ depend only on $S$. We have thus obtained the desired inequality for every $\mathfrak{p}\in\cF$.
\end{proof}

The following theorem is a slight generalisation of Theorem \ref{sys:thm} in the sense that it also includes closed $2$-dimensional hyperbolic orbifolds.

\begin{theorem}
Let $\Gamma < \PSLTR$ be a cocompact semi-arithmetic Fuchsian group and let $K$ be its invariant trace field. Then, for infinitely many prime ideals $\mathfrak{p} \subset \mathcal{O}_K$, the corresponding congruence subgroups $\Gamma(\mathfrak{p}) < \Gamma$ are torsion free and the closed Riemann surfaces $S_\mathfrak{p}=\Gamma(\mathfrak{p})  \backslash \Hy$ satisfy

\[\sys(S_{\mathfrak{p}}) \geq C \log(\area(S_{\mathfrak{p}}))- c,\]
where $C>0, c  \in \R$ are constants that do not depend on $\mathfrak{p}$.
\end{theorem}

\begin{proof}
	Take the family of prime ideals obtained in the Theorem \ref{quotientcongruence}, so that, for any prime ideal $\mathfrak{p}$ in this family, it holds that $\Gamma^{(2)}/\Gamma^{(2)}(\mathfrak{p})$ is isomorphic to $\PSL(2,\mathcal{O}_K/\mathfrak{p} \mathcal{O}_K)$ and, in particular, we may apply Corollary \ref{indexgrowth}.

	Let $D=\{ \sum_j x_j B_j \mid x_j \in K, \, B_j \in \tilde{\Gamma}^{(2)} \}$ be the invariant quaternion algebra of $\Gamma$ and let $\mathcal{Q}=\{ \sum_j \alpha_j B_j \mid \alpha_j \in \mathcal{O}_K , \, B_j \in \tilde{\Gamma}^{(2)} \}$ be an order in $D$. If we consider a nontrivial element $\gamma \in \Gamma(\mathfrak{p})$, then for some preimage $A \in \mathcal{Q}^1$ of  $\gamma$ we may write
	 
	\[A^2 - 1 = \sum_{i} \alpha_i B_i \mbox{ with } \alpha_i \in \mathfrak{p} \mbox{ and } B_i \in \tilde{\Gamma}^{(2)}.\]
	Therefore,
	
	\[\det(A^2-1)=\left(\sum_i \alpha_i B_{i}\right)\left(\displaystyle\sum_{j} \alpha_j B^*_j\right)= \sum_{i} \alpha_i^2 \det(B_i) + \sum_{i<j} \alpha_{i} \alpha_{j} \tr(B_iB^*_j)  \in \mathfrak{p}^{2}.\]
	where the superscript $*$ denotes the usual involution in the quaternion algebra $D$. On the other hand, it follows from elementary identities that $\det(A^2 - 1)=2-\tr(A^2)= 4 - \tr^2(A) = 4 - \Tr^2(\gamma)$. So that $4 - \Tr^2(\gamma) \in \mathfrak{p}^2$.
	
	Note that if $\alpha \in \mathfrak{p}^2$ is nonzero, then $\Ht(\alpha) \ge N(\mathfrak{p})^{\frac{2}{d}}$. Indeed, let $v_\mathfrak{p} \in V_f(K)$ be the valuation corresponding to $\mathfrak{p}$. By definition of $v_\mathfrak{p}$ we have that $|\alpha|_\mathfrak{p} \leq N(\mathfrak{p})^{-2}$. Since $\alpha$ is an algebraic integer, $v(\alpha) \leq 1$ for any $v \in V_f(K)$. So, if we define $V_\alpha(K)=\{v \in V_\infty(K) \mid v(\alpha)>1 \}$, then
	\[\Ht(\alpha)^d=\prod_{v \in V_\alpha(K)}v(\alpha).\]
	The \emph{product formula} (see \cite[Chapter III (1.3)]{Neukirch99}), however, implies that
	\[\prod_{v \in V(K)}v(\alpha)=1,\]
	and therefore 
	
	\begin{align*}
	\Ht(\alpha)^d  =  \left( \prod_{v \in V(K) \setminus V_\alpha(K)} v(\alpha) \right)\inv \geq \vp(\alpha)\inv \geq N(\mathfrak{p})^2.
	\end{align*}
	
	
	
	Returning to the arbitrarily chosen $\gamma \in \Gamma \setminus \{1\}$, we have that $4 - \Tr(\gamma)^2 \neq 0 $ since $\gamma$ is not parabolic. Then,
	\begin{align*}
	  \Ht(4 - \Tr^2(\gamma)) \geq N(\mathfrak{p})^{\frac{2}{d}}.
	\end{align*}
	
	By applying the properties of the height function in Subsection \ref{numbertheory}, we observe that
	\begin{align*}
	\Ht(4 - \Tr^2(\gamma)) \leq 4 \, \Ht(4) \, \Ht(\Tr^2(\gamma)),
	\end{align*}
	which leads to
	\begin{align}\label{heightlowbound}
	\Ht(\Tr^2(\gamma)) \geq \frac{1}{16}N(\mathfrak{p})^{\frac{2}{d}},
	\end{align}
	(note that $\Ht(4) = 4$).
	
	Now, if $\gamma$ is elliptic, then $\Tr(\gamma)=|2\cos(\frac{\pi}{n})|$ where $n$ is the order of $\gamma$. Since $2\cos(\frac{\pi}{n})=\zeta_{2n}+\zeta_{2n}\inv$ where $\zeta_{2n}$ is the $2n$-th primitive root of unit, any Galois conjugate of $2\cos(\frac{2\pi}{n})$ has absolute value at most $2$, i.e. $\Ht(\Tr^2(\gamma)) \leq 4$. Hence, it follows from \eqref{heightlowbound} that, for any prime ideal $\mathfrak{p}$ satisfying $N(\mathfrak{p})>8^{d}$ the congruence subgroup $\Gamma(\mathfrak{p})$ is torsion free. Henceforth, we will only consider prime ideals satisfying this additional property. Note that it remains an infinite family of prime ideals.
	
	In order to estimate from below the displacement of any element in $\Gamma(\mathfrak{p})$ it is sufficient to estimate the displacement of its square. So we will give a lower bound for the displacement of any element in $\Gamma^{(2)}(\mathfrak{p})$. Moreover, we need to relate this lower bound with the growth of the area of the corresponding coverings. By Lemma \ref{milnorlemma} and Lemma \ref{areagrowth}, it is sufficient to show that, for our family of prime ideals $\mathfrak{p}$, any nontrivial element in the corresponding congruence group has word length at least $c_1\log(N(\mathfrak{p}))-c_2$ for some constants $c_1>0$ and $c_2$ that may depend on $\Gamma$ but not on the prime ideal $\mathfrak{p}$. For the remainder of the proof, we will denote such constants indistinctly by $C>0$ and $c$.
	
	
	We want to estimate the word length $w_\Sigma$ of a nontrivial element $\eta \in \Gammat(\mathfrak{p})$, with respect to a canonical set of generators $\Sigma$ of $\Gammat$. There exists a finite extension $L$ of $K$ of degree at most $2$ for which we may suppose that, up to conjugation, $\Gammat \subset \PSL(2,L)$ (cf. \cite[Corollary 3.2.4]{gtmReid03}).  Let $\tau=\max \{\Ht(\sigma) \mid \sigma \in \Sigma\}$. By definition, $\Ht(\sigma)=\Ht(\sigma\inv).$ If $\eta=\sigma^{i_1} \cdots \sigma^{i_w}$ with $i_j \in \Z$ is a word of minimal length $w_\Sigma(\eta)$, then we may apply property \eqref{prodH} of the height function (see Subsection \ref{numbertheory}) successively in order to obtain $\Ht(\eta) \leq (4\tau)^{w_\Sigma(\eta)-1}\tau$, i.e.
	\begin{align} \label{wordlowbound}
		w_\Sigma(\eta) \geq C\log(\Ht(\eta))  - c.
	\end{align}
	It follows also from the definitions that $\Ht(\Tr(\eta)) \leq 4 \Ht(\eta)$. Finally, \eqref{heightlowbound} and \eqref{wordlowbound} together imply that
	
	\[w_\Sigma(\eta) \geq C\log(N(\mathfrak{p})) - c.\]

\end{proof}

Let $S=\Gamma \backslash \Hy$ be a semi-arithmetic hyperbolic surface, and let $K$ be the invariant trace field of $S$. Each Galois embedding $\phi: K \rightarrow \mathbb{R}$  can be extended to a monomorphism $\Phi: \Gamma \rightarrow \PSLTR$ given by applying a suitable  extension of $\phi$ on the entries of the element of $\Gamma$. This map is well defined up to conjugation by elements of $\PGLTR$. There exists a choice of monomorphisms $\Phi_1, \dots, \Phi_r$ with $r \leq [K:\mathbb{Q}]$ such that the map $\Phi=(\Phi_1,\cdots,\Phi_r):\Gamma \rightarrow \PSLTR^r$  
satisfies $\Phi(\Gamma) \subset \Lambda$ where $\Lambda$ is an irreducible arithmetic lattice acting on $\Hy^r$ (see \cite[Remark 4]{SW}).

We say that $S$ \emph{admits modular embedding} if there exists a holomorphic map $F:\Hy \rightarrow \Hy^r$ such that 
$$F(\gamma(z)) = \Phi(\gamma)(F(z)) \mbox{ for all } z \in \Hy \mbox{ and } \gamma \in \Gamma. $$
In this case, the map $F$ induces a $\pi_1$-injective embedding $f: S \rightarrow M$ where $M=\Lambda \backslash \Hy^r$ is an arithmetic orbifold locally isometric to $\Hy^r.$

Now we are ready to prove Theorem \ref{modemb:thm}. We repeat its statement for the convenience of the reader. Recall that $\sys_1$ denotes the 1-systole of a manifold, i.e, the infimum of the lengths of non-trivial closed curves. In the case of surface, we denote its systole simply by $\sys$.

\begin{repthrm}[\ref{modemb:thm}]
	If $S$ is a semi-arithmetic Riemann surface which admits modular embedding into an arithmetic $X_r$-manifold for some $r \ge 1$, then $S$ admits a sequence of congruence coverings $S_i \to S$ of degree arbitrarily large satisfying
	$$\sys(S_i) \geq \frac{4}{3r} \log(\area(S_i))-c,$$
	where the constant $c$ does not depend on the $i$.
\end{repthrm}

\begin{proof}
Let $S=\Gamma \backslash \Hy$ be a closed semi-arithmetic surfaces admitting a modular embedding, i.e., there exists an irreducible arithmetic lattice $\Lambda < \PSLTR^r$ defined over the invariant trace field $K$ of $\Gamma$ and a holomorphic map $F: \Hy \rightarrow \Hy^r$ such that the monomorphism  $\Phi=(\Phi_1,\cdots,\Phi_r):\Gamma \rightarrow \PSLTR^r$ satisfies $\Phi(\Gamma) \subset \Lambda$ and $$F(\gamma(z)) = \Phi(\gamma)(F(z)) \mbox{ for all } z \in \Hy. $$

We may suppose (after possibly replacing $\Lambda$ and $\Gamma$ by finite index subgroups) that $\Lambda$ is contained in an arithmetic lattice arising from a quaternion algebra.
Then, if the norm of the ideal $I  \subset \mathcal{O}_K$ is sufficiently large, by Proposition \ref{systvsnorm} the subgroup $\Lambda(I)$ must be torsion free and 
\begin{equation} \label{sysgrowr}
\sys_1(\Lambda(I) \backslash \Hy^r) \ge \frac{4}{\sqrt{r}}\log(\mathrm{N}(I))-c_1,
\end{equation}
for some constant $c_1>0$ which does not depend on $I$.

By the Schwarz-Pick Lemma, any holomorphic map is $1$-Lipschitz, with respect to the hyperbolic metric. Therefore, the induced $\pi_1-$injective map $F:S \rightarrow \Lambda \backslash \Hy^r$ is $\sqrt{r}$-Lipschitz.

Now we consider the family $\mathcal{F}$ of prime ideals of $\mathcal{O}_K$ given in the Theorem \ref{quotientcongruence}. For each ideal $\mathfrak{p} \in \mathcal{F}$ take the congruence subgroup of $\Gamma(\mathfrak{p})$. Then the map $F$ induces an embedding $F_{\mathfrak{p}}: \Gamma(\mathfrak{p}) \backslash \Hy \rightarrow \Lambda(\mathfrak{p}) \backslash \Hy^r$ which is also $\sqrt{r}$-Lipschitz. We want to compare the systoles of these manifolds, but some care needs to be taken here. If $\gamma \in \Gamma(\mathfrak{p})$ is hyperbolic we must first guarantee that $\Phi(\gamma)$ is a hyperbolic translation in $\Lambda$  (see Appendix \ref{pliniogeral}), i.e. that no Galois conjugate of $\Tr(\gamma)$ is equal to $\pm2$. Indeed, if that was the case, then all Galois conjugates of $\Tr(\gamma)$ would be equal to $\pm 2$ and $\gamma$ would be parabolic, contrary to our assumptions. Thus, we have the following inequality:
\[\sys_1(\Lambda(\mathfrak{p}) \backslash \Hy^r) \leq \sqrt{r} ~ \sys(\Gamma(\mathfrak{p}) \backslash \Hy) ,\]
which, together with \eqref{sysgrowr}, implies that
\[\sys(\Gamma(\mathfrak{p}) \backslash \Hy) \ge \frac{4}{r}\log(\mathrm{N}(\mathfrak{p}))-c_2\]
for every $\mathfrak{p}$ contained in an infinite subfamily $\mathcal{F}' \subset \mathcal{F}$, where $c_2>0$ does not depend on $\mathfrak{p}$.

By Lemma \ref{areagrowth}, if we define the congruence coverings $S_{\mathfrak{p}}=\Gamma(\mathfrak{p}) \backslash \Hy$ for each $\mathfrak{p} \in \mathcal{F}'$, we conclude that $\area(S_{\mathfrak{p}}) \to \infty$ when $N(\mathfrak{p}) \to \infty$ and

\[\sys(S_{\mathfrak{p}}) \ge \frac{4}{3r}\log(\area(S_{\mathfrak{p}}))-c\]
for some constant $c>0$ that does not depend on $\mathfrak{p}$.
\end{proof}

\appendix
\section{}\label{procedemeister}

In this appendix we use the Reidemeister process in order to give a standard presentation for the group $\mathcal{K} \unlhd \Gamma$ in terms of the generators of $\Gamma$, as defined in Section \ref{densetoflengths}. For more information on the Reidemeister process we refer the reader to \cite{Bogopolski08}. Recall that
\begin{align*}
\Gamma=	\langle c_1,\dots, c_6 \mid c_1^2 = \dots = c_6^2 = c_1\cdots c_6 = 1\rangle.
\end{align*}
Let $\theta\!:\Gamma \to \Z/2\Z$ be the epimorphism defined by $c_i \mapsto 1$, $i=1,\dots,6$ and define $\mathcal{K}:=\ker\theta \unlhd \Gamma$.

Let  $\phi\!: F_6=F(x_1,\dots,x_6) \to \Gamma$ be an epimorphism from the free group of rank $6$ with generators $\{x_1,\dots, x_6 \}$ onto $\Gamma$ given by $\phi(x_j)=c_j$, $j=1,\dots,6$. Finally, let $\widetilde{\cK}$ be the pre-image of $\mathcal{K}$ with respect to $\phi$. Since $[\Gamma : \mathcal{K}] = 2$ and $\phi$ is surjective, it follows that $\widetilde{\cK}$ also has index 2 in $F_6$. We pick the set $\fT = \{1, x_1\}$ as a Schreier transversal for $\widetilde{\cK}$ in $F_6$. As usual, for any $g\in F_6$ we denote by $\overline{g}$ the (unique) element in $\fT$ with the property that $\widetilde{\cK} g = \widetilde{\cK} \overline{g}$. Then, by the Reidemeister process, the following elements generate the free group $\widetilde{\cK}$:
\begin{align*}
1\cdot x_i \cdot (\overline{1\cdot x_i})\inv = x_ix_1\inv , \quad i=2,\dots,6 ;\\
x_1x_j \cdot (\overline{x_1x_j})\inv = x_1x_j , \quad j=1,\dots, 6.
\end{align*}
Let us rename this generators as:
\begin{align}\label{gen:xtoy}
y_j = x_1x_j,\ j=1,\dots,6 \quad \text{and} \quad y_{5+i} = x_ix_1\inv,\ i=2\dots,6.
\end{align}

In order to find the defining relations we rewrite the words $trt\inv$, where $t\in\{1, x_1 \}$ and $r\in\{x_1^2,\dots,x_6^2, \,  x_1\cdots x_6 \}$, in terms of $\{y_1,\dots,y_{11} \}$:

\begin{equation*}
\setlength{\jot}{20pt}
\begin{split}
&x_j^2 = (x_jx_1\inv)(x_1x_j) =
\begin{cases}
y_1 , &\text{ if } j=1;\\
y_{5+j}\,y_j, &\text{ if } j=2,\dots,6;
\end{cases} \\
&x_1\cdots x_6 = (x_1x_2)(x_3x_1\inv)(x_1x_4)(x_5x_1\inv)(x_1x_6) = y_2 \, y_8 \, y_4 \, y_{10} \, y_6 \, ;\\
&x_1x_j^2x_1\inv = (x_1x_j)(x_jx_1\inv) =
\begin{cases}
y_1 , &\text{ if } j=1;\\
y_j \, y_{5+j}, &\text{ if } j=2,\dots,6;
\end{cases}\\
&x_1(x_1\cdots x_6)x_1\inv = x_1^2(x_2x_1\inv)(x_1x_3)(x_4x_1\inv)(x_1x_5)(x_6x_1\inv) = y_1 \, y_7 \, y_3 \, y_9 \, y_5 \, y_{11}.
\end{split}
\end{equation*}

Note that we may eliminate the generators $y_1, \, y_7, \, y_8, \, y_9, \, y_{10}, \, y_{11}$ and obtain the following presentation:

\begin{align}\label{31120.2}
\cK = \langle  y_2, y_3, y_4, y_5, y_6 \mid y_2 \, y_3\inv \, y_4 \, y_5\inv \, y_6 \ , \  y_2\inv \, y_3 \, y_4\inv \, y_5 \, y_6\inv \rangle.
\end{align}

In order to make it less cumbersome, let us once again rename the generators, now as

\begin{align*}
y_2 =: a, \ y_3\inv =: b, \ y_4 =: c, \ y_5\inv =: d.
\end{align*}

The two relations in \eqref{31120.2} imply that 
\begin{align*}
&y_6=y_2\inv\,y_3\,y_4\inv\,y_5 = (y_5\inv \, y_4 \, y_3\inv \, y_2)\inv;\\
&y_2 \,y _3\inv \, y_4 \, y_5\inv \, y_6 = 1.
\end{align*}
which, in terms of the new names introduced above, yield the following presentation
\begin{align*}
\cK = \langle a,b,c,d \mid abcd(dcba)\inv \rangle.
\end{align*}

Finally, we achieve the standard presentation by making one last change in the generators of $\cK$:
\begin{equation*}
x:= ab^{-2}, ~ y:= b, ~ x':= bac, ~ y':= dc.
\end{equation*}

Note that $\{x,y,x',y'\}$ generates $\mathcal{K}$ indeed, since
\begin{align*}
x\,y^2 = a,~ y= b,~ y^{-2}\,x^{-1}\,y^{-1}\,x'= c,~ y'\,(x')\inv\, y\,x\,y^2 = d.
\end{align*}
With this new set of generators, the relation 	$abcd(dcba)\inv=1$ becomes $[x,y][x',y']=1$, i.e.
\[\mathcal{K}= \langle x,y,x',y' \mid [x,y][x',y']\rangle.\]

By the definition of the homomorphism $\phi$ and \eqref{gen:xtoy}, we conclude the appendix pointing out that:
\begin{align*}
c_1c_2 = \phi(x_1x_2) = \phi(y_2) \in \cK
\end{align*}

\section{}\label{pliniogeral}
In this appendix, we prove the following theorem. 

\begin{theorem} \label{sysvsvol}
Let $M$ be a closed arithmetic $X_r$-orbifold for some $r \ge 1$, defined over a number field $K$. There exists a finite covering $M' \to M$ such that for all but finitely many ideals $I$ of $\mathcal{O}_K$, the congruence covering $M'_I \to M'$ is a manifold. Furthermore, $\vol(M'_I) \to \infty$ when $N(I) \to \infty$ and $$\sys_1(M'_I) \ge \frac{4}{3\sqrt{r}}\log(\vol(M'_I)) -c,$$
where $c$ does not depend on $I$.
\end{theorem}

Note that Theorem \ref{plinio:appendix} follows immediately. Theorem \ref{sysvsvol} generalises the noncompact case proved by P. Murillo in \cite{Murillo17}. In fact, we follow his idea and refer to [op. cit] for more details.

\subsection{The geometric structure}
Let $\Hy^r$ be the product space of $r$ copies of the hyperbolic plane $\Hy$ endowed with the product metric. The space $\Hy^r$ is a simply connected symmetric Riemannian manifold of nonpositive curvature. Consider the group $G=\PSL(2,\mathbb{R})^r$ acting on $\Hy^r$ in the following way: given a point $z=(z_1,\ldots,z_r) \in \Hy^r$ and $g=(g_1,\ldots,g_r) \in G$, define $g(z)=(g_1(z_1),\ldots,g_r(z_r))$. Note that this action preserves the Riemannian metric on $\Hy^r$. Therefore, $G$ is a group of isometries of $\Hy^r$ whose action is continuous and transitive.

For any $g \in G$ nontrivial we define its displacement $\ell(g)$ by $$\ell(g)=\inf\{\mathrm{d}((g_1(z_1),\ldots,g_r(z_r)),(z_1,\ldots,z_r)) ~|~ (z_1,\ldots,z_r) \in \Hy^r \},$$
where $$\mathrm{d}((x_1,\ldots,x_r),(y_1,\ldots,y_r)):=\displaystyle \sqrt{d_\Hy(x_1,y_1)^2+\cdots+d_\Hy(x_r,y_r)^2}.$$

We classify the nontrivial elements of $G$ in terms of their geometric action. 
\begin{enumerate}
    \item We say that $g$ is \emph{elliptic} if there exists $z=(z_1,\ldots,z_n) \in \Hy^r$ such that $g(z)=z$;
    \item We say that $g$ is \emph{parabolic} if $g$ is not elliptic and $\ell(g)=0;$
    \item We say that $g$ is \emph{hyperbolic} if $\ell(g)>0;$ 
    \item We say that a hyperbolic element $g \in G$ is a \emph{hyperbolic translation} if there exists a complete geodesic $\xi:\mathbb{R} \rightarrow \Hy^r$ with unit speed and such that $$g(\xi(t))=\xi(t+\ell(g)) \mbox{ for all } t \in \mathbb{R}.$$
    In this case, we say that the geodesic $\xi$ is the \emph{axis} of $g$. It is worth noting that every component of a hyperbolic translation is either elliptic or hyperbolic.
\end{enumerate}

If $\Gamma < G$ is a discrete group, then the action of $\Gamma$ in $\Hy^r$ is properly discontinuous. In order to produce a manifold covered by $\Hy^r$ it is necessary and sufficient that $\Gamma$ is torsion free (i.e., $\Gamma$ does not contain elliptic elements). In general, if $\Gamma$ is discrete, the quotient space $\Gamma \backslash \Hy^r$ is an $r-$dimensional orbifold locally isometric to $\Hy^r$. 

If $g \in \Gamma$ is a hyperbolic translation, then its axis projects in the orbifold $\Gamma \backslash \Hy^r$ as a closed geodesic of length $\ell(g)$. Conversely, any closed geodesic on $\Gamma \backslash \Hy^n$ is obtained in this manner. We define the systole of an orbifold $M$ as the minimum of the set of lengths of all closed geodesics in $M$. This number is denoted by $\sys_1(M)$.  

For any hyperbolic translation $g=(g_1,\ldots,g_r) \in G$ we can write $\{1,\cdots,r\}=H(g) \cup E(g)$, where $|\tr(g_h)|>2$ if $h \in H(g)$ and $|\tr(g_e)|<2$ if $e \in E(g)$. In this case we have (cf. Proposition \ref{31120.1})
\begin{equation}\label{dispform}
\ell(g)=2\sqrt{ \sum_{h \in H(g)} \left[\cosh^{-1}\left(\frac{|\tr(g_h)|}{2} \right)\right]^2}.    
\end{equation}

\subsection{The arithmetic structure} \label{aritructure}
The cocompact irreducible arithmetic lattices in $G$ are constructed as follows. Let $K$ be a totally real number field of degree $d$ and let $A$ be a quaternion algebra over $K$ that splits at exactly $r<d$ Archimedean places of $K$. There exists a surjective algebra morphism $\rho:A \otimes_K \mathbb{R} \rightarrow \mathrm{M}_2(\mathbb{R})^r$ with compact kernel isomorphic to the product of $d-r$ copies of the Hamiltonian algebra. Let $\mathcal{Q}$ be an order of $A$ and consider the subgroup $\mathcal{Q}^1$ of elements in $\mathcal{Q}$ of reduced norm 1. The image $\rho(\mathcal{Q}^1) < G$ is a cocompact irreducible arithmetic lattice. 
Conversely, by Margulis Arithmeticity Theorem, if $\Gamma < G$ is a cocompact irreducible lattice, then $\Gamma$ is commensurable with $\rho(\mathcal{Q}^1)$ for some $K,A,\rho$ and $\mathcal{Q}$ as above (cf. \cite{Bor81}).
Once these parameters are fixed, we define what it means to be a congruence covering in this setting. 

Choose a standard basis $1,\iota,j,\kappa$ for the quaternion algebra $A$ over $K$, i.e., a basis such that $\iota^2=a, \, j^2=b, \, \iota j=\kappa = -j\iota$ with $a,b \in \mathcal{O}_K\minus\{0\}$, and fix the order $$\mathcal{Q}=\{\alpha_0+\alpha_1\iota +\alpha_2j+\alpha_3\kappa \in A \mid \alpha_i \in \mathcal{O}_K\}.$$
With such basis we recall that the \emph{trace} and \emph{reduced norm} of an element  $\alpha=\alpha_0+\alpha_1\iota +\alpha_2j+\alpha_3\kappa \in A$ are given by $$\tr(\alpha)=2\alpha_0 \quad \mbox{ and } \quad \mathrm{n}(\alpha)=\alpha_0^2
-a\alpha_1^2-b\alpha_2^2+ab\alpha_3^2.$$

For $i=1,\ldots,r$, let $\sigma_i:K \rightarrow \mathbb{R}$ denote an embedding of $K$ over which $A$ splits. We may assume, without loss of generality, that $\sigma_1=\mathrm{id} : K \rightarrow \mathbb{R}$ is the trivial inclusion of $K$ in $\mathbb{R}$. Consider $\rho:A \otimes_K \mathbb{R} \rightarrow \mathrm{M}_2(\mathbb{R})^r$ with components $\rho=(\rho_1,\ldots,\rho_r)$ satisfying 
\begin{equation}\label{tracesrelation}
\tr(\rho_i(\alpha))=\sigma_i(\tr(\alpha)), \quad \mbox{ for each } i=1,\ldots,r.  
\end{equation}

Let $\tau_1,\cdots,\tau_l:K \rightarrow \mathbb{R}$ be the embeddings at which $A$ ramifies, i.e., for each $j=1,\ldots,l$ the quaternion algebra $\left(\frac{\tau_j(a) \, , \, \tau_j(b)}{\mathbb{R}} \right)$ satisfies $\tau_j(a),\tau_j(b)<0$ (in other words, it is isomorphic to the division algebra $\cH$ of Hamilton's quaternions). If $\alpha \in \mathcal{Q}^1$ and $\alpha \neq \pm 1$, then applying these embeddings to the equation $1=\alpha_0^2 -a\alpha_1^2-b\alpha_2^2+ab\alpha_3^2,$ yields

\[\tau_j(\alpha_0^2)<1 \mbox{ for any } k=1,\ldots,l. \]

Therefore,

\begin{equation}\label{taus}
|\tau_j(\tr(\alpha))| = 2|\tau_j(\alpha_0)| \leq 2,     
\end{equation}

for all $j=1,\ldots,l$ and $\alpha \in \mathcal{Q}^1.$

For any ideal $I \subset \mathcal{O}_K$ we define the following group

\[ \mathcal{Q}^1(I)=\{\alpha=\alpha_0+\alpha_1\iota+\alpha_2j+\alpha_2\kappa \in \mathcal{Q}^1 \mid \alpha_0-1,\alpha_1,\alpha_2,\alpha_3 \in I \}.\]

The group $\mathcal{Q}^1(I)$ is the congruence subgroup of $\mathcal{Q}^1$ of level $I$. We can now prove the following useful proposition.

\begin{prop}\label{systvsnorm}
Consider $K,A,\rho, \mathcal{Q}$ as above and let $I \subset \mathcal{O}_K$ be an ideal with norm $N(I)$ satisfying $\mathrm{N}(I) \geq 2^d$. Then $\rho(\mathcal{Q}^1(I))$ is a torsion free lattice. Furthermore, we have the following estimate for the systole of the manifold $M_I=\rho(\mathcal{Q}^1(I)) \backslash \Hy^r$ :

\[\sys(M_I) \ge \frac{4}{\sqrt{r}}\log(\mathrm{N}(I))-c,\]

where $c>0$ is a constant that does not depend on $I$.
\end{prop}

\begin{proof}
For $\alpha=\alpha_0+\alpha_1\iota +\alpha_2j+\alpha_3\kappa \in \mathcal{Q}^1$,  $g_\alpha=\rho(\alpha)$ cannot be parabolic since $\rho(\mathcal{Q}^1)$ is cocompact. In particular, $\tr(\alpha) = \tr(\rho_1(\alpha)) \neq \pm 2$ which implies that $\alpha_0 \neq \pm 1$. So $g_\alpha$ must be either elliptic or hyperbolic. A necessary (and sufficient) condition for $g_\alpha$ to be elliptic is that $|\tr(\rho_i(\alpha))|<2$ for all $i=1,\ldots,r.$ By \eqref{tracesrelation} and \eqref{taus} we see that

\begin{equation}\label{25320.1}
\begin{split}
|\mathrm{N_{K\mid\Q}}(\tr(\alpha)-2)| &= \prod_{i=1}^r |\sigma_i(\tr(\alpha)) - 2| \prod_{k=1}^{d-r}|\tau_k(\tr(\alpha)) - 2|\\
                                      &\leq \prod_{i=1}^r (|\sigma_i(\tr(\alpha))| + 2) \prod_{k=1}^{d-r}(|\tau_k(\tr(\alpha))| + 2)\\
                                      &< 2^{2d}.  
\end{split}
\end{equation}

Here, $\mathrm{N}_{K\mid\Q}$ denotes the (field) norm of the extension $K\mid\Q$. We recall that, for a nonzero algebraic integer $x\in\cO_K$, $|\mathrm{N}_{K\mid\Q}(x)|$ is equal to  the norm of the ideal $x\cO_K$, denoted $\mathrm{N}(x)$. As observed above, $\tr(\alpha)-2 \neq 0$.

Now, if $\alpha \in \mathcal{Q}^1(I),$ it follows from 
\[\mathrm{n}(\alpha)=1 \mbox{ and } \alpha_0-1,\alpha_1,\alpha_2,\alpha_3 \in I\]
that

\[\tr(\alpha) - 2 = 2\alpha_0 - 2 =(\alpha_0^2 - 1) -(\alpha_0-1)^2 = a\alpha_1^2+b\alpha_2^2-ab\alpha_3^2-(\alpha_0-1)^2 \in I^2\]

so that, in particular,

\begin{equation}\label{tracong}
 \mathrm{N}(\tr(\alpha)-2) \ge \mathrm{N}(I)^2.   
\end{equation}

Hence, if $\rho(\mathcal{Q}^1(I))$ contains a torsion element, it follows from \eqref{25320.1} and \eqref{tracong} that

\[\mathrm{N}(I) < 2^d.\]

This proves the first part of the theorem. From now on we will assume that $\mathrm{N}(I) \geq 2^d$. We must estimate the displacement of a hyperbolic translation in $\rho(\mathcal{Q}^1(I))$.

If $\rho(\alpha) \in \rho(\mathcal{Q}^1(I))$ is one such hyperbolic translation, then we  may write $\{1,\ldots,r\}$ as a disjoint union $H(\alpha) \cup E(\alpha)$ where $|\tr(\rho_h(\alpha))|>2$ if $h \in H(\alpha)$ and $|\tr(\rho_e(\alpha))|<2$ if $e \in E(\alpha)$. By \eqref{dispform} and \eqref{tracesrelation} we have that

\[\ell(\rho(\alpha)) = 2\sqrt{\sum_{h \in H(\alpha)} \left[\cosh^{-1}\left(\frac{|\sigma_h(\tr(\alpha))|}{2} \right)\right]^2} \geq 2 \sqrt{ \sum_{h \in H(\alpha)} \log^2\left(\frac{|\sigma_h(\tr(\alpha))|}{2}\right)},\]
where the inequality follows from the observation that $\cosh^{-1}(x)=\log(x+\sqrt{x^2-1}) \ge \log(x)$ for all $x>1.$ By means of the convexity of the quadratic function, we obtain that

\[\ell(\rho(\alpha)) \geq \frac{2}{\sqrt{\# H(\alpha)}} \displaystyle \sum_{h \in H(\alpha)} \log \left(\frac{|\sigma_h(\tr(\alpha))|}{2}\right) \ge \frac{2}{\sqrt{r}}\log\left( \displaystyle \prod_{h \in H(\alpha)}|\sigma_h(\tr(\alpha))|2^{-\# H(\alpha)} \right).\]

On the other hand

\begin{align*}
|\mathrm{N}(\tr(\alpha)-2)| &\leq \prod_\sigma (|\sigma(\tr(\alpha))| + 2)\\
                                          &\leq 2^d \prod_\sigma \max\{|\sigma(\tr(\alpha))|,2\}\\
                                          &= 2^d \, 2^{d-\# H(\alpha)} \prod_{h \in H(\alpha)}|\sigma_h(\tr(\alpha))|.
\end{align*}

Thus, for any hyperbolic translation $\rho(\alpha)$ with $\alpha \in \mathcal{Q}^1(I)$, the following holds:

\[\ell(\rho(\alpha)) \geq \frac{2}{\sqrt{r}}\log\left(\frac{|\mathrm{N}(\tr(\alpha)-2)|}{2^{2d}}\right)   \ge \frac{2}{\sqrt{r}}\log(|\mathrm{N}(\tr(\alpha)-2)|)-\frac{4d\log(2)}{\sqrt{r}}.\]

We conclude the proof by using inequality \eqref{tracong}. Indeed, if we define $c:=\frac{4d\log(2)}{\sqrt{r}}$, then $c$ does not depend on $I$ and, for any closed geodesic $\beta$ on $M_I=\rho(\mathcal{Q}^1(I)) \backslash \Hy^r$, there exists $\alpha \in \mathcal{Q}^1(I)$ as above such that the axis of $\rho(\alpha)$ projects onto $\beta$ and
\[\ell(\rho(\alpha)) \ge \frac{4}{\sqrt{r}}\log(\mathrm{N}(I))-c.\]
Since $\beta$ was arbitrary, the theorem follows.
\end{proof}

\begin{proof}[Proof of Theorem \ref{sysvsvol}]
Let $M=\Gamma \backslash \Hy^r$ be a closed arithmetic orbifold, we have seen that $\Gamma$ is commensurable with $\rho(\mathcal{Q}^1)$ for some $K,A,\rho$ and $\mathcal{Q}$ as in subsection \ref{aritructure}. Thus, $\Gamma$ contains a finite index subgroup $\Gamma'$ such that $\Gamma' < \rho(\mathcal{Q}^1)$. Let $M'=\Gamma' \backslash \Hy^r$ be the corresponding finite covering of $M$. If we consider the sequence of ideals of $\mathcal{O}_K$ satisfying the hypothesis in Proposition \ref{systvsnorm} then the corresponding lattices  $\Gamma'(I)=\Gamma' \cap \rho(\mathcal{Q}^1(I))$ are torsion free and define congruence coverings  $M'_I:=\Gamma'(I) \backslash \Hy^n$ of $M'$ with volume 
\begin{equation*}
\vol(M'_I)=[\Gamma':\Gamma'(I)]\vol(M').    
\end{equation*}
Therefore, the asymptotic growth of $\vol(M'_I)$ depends only on the growth of $[\Gamma':\Gamma'(I)],$ where $I$ varies on the set of ideals of $\mathcal{O}_K$ with $\mathrm{N}(I) > 2^{d}.$

Since $\mathcal{Q}^1 \cap \mathrm{ker}(\rho)$ is finite with, say, $\nu$ elements, then  \[[\Gamma':\Gamma'(I)] \leq [\rho(\mathcal{Q}^1):\rho(\mathcal{Q}^1(I))] = \frac{[\mathcal{Q}^1:\mathcal{Q}^1(I)]}{\nu}.\]
  
In \cite[Corollary 4.6]{KSV07} it is proved that there exists a constant $\lambda=\lambda(A,\mathcal{Q})>0$ such that 
\begin{equation*}
[\mathcal{Q}^1:\mathcal{Q}^1(I)]  \leq \lambda N(I)^3.    
\end{equation*}

On the other hand, by Proposition \ref{systvsnorm} we have that
\begin{equation*}
 \sys_1(M'_I) \geq \sys_1(\rho(\mathcal{Q}^1(I)) \backslash \Hy^r) \ge \frac{4}{\sqrt{r}}\log(\mathrm{N}(I))-c_1,   
\end{equation*}
for some constant $c_1>0$ which does not depend on $I$.

Therefore,
$$\sys_1(M'_I) \ge \frac{4}{3\sqrt{r}}\log(\vol(M'_I))-c,$$
where $c$ does not depend on $I$.

Finally, by Proposition \ref{systvsnorm}, $\sys_1(M'_I)\to\infty$ when $N(I)\to\infty$. Since $\log(\vol(M'_I))\geq C' \sys_1(M'_I) - c'$, it follows that $\vol(M'_I)\to\infty$ as $N(I)\to\infty$.

\end{proof}


\end{document}